\documentclass[11pt, english, a4paper]{amsart}

\usepackage{graphicx}
\usepackage{amsmath}
\usepackage{amssymb}

\usepackage{bm} 
\usepackage[T1]{fontenc}
\usepackage{amsfonts}
\usepackage{mathrsfs}
\usepackage{yfonts}
\usepackage{url}
\usepackage[osf]{mathpazo}  

\usepackage{mathtools}

\usepackage{graphicx}
\usepackage{verbatim}
\usepackage{enumerate}	
\usepackage[english]{babel}
\usepackage{ebproof}		
\usepackage{quoting}
\usepackage{array,booktabs}
\usepackage[all]{xy}\SelectTips{cm}{}
\usepackage{pifont}

\usepackage[shortlabels]{enumitem}
\usepackage{units}
\usepackage{xspace}\xspaceaddexceptions{\,}

\usepackage{booktabs,caption}
\usepackage{tikz-cd}
\quotingsetup{font=small}	

\usepackage{todonotes}

\theoremstyle{definition}
\newtheorem{definition}{Definition}

\newtheorem{remark}[definition]{Remark}

\theoremstyle{plain}
\newtheorem{lemma}[definition]{Lemma}
\newtheorem{proposition}[definition]{Proposition}
\newtheorem{theorem}[definition]{Theorem}
\newtheorem{corollary}[definition]{Corollary}

\newcommand\A{{\mathbf A}}

\newcommand\Fm{\mathbf{Fm}}

\newcommand\WK{{\mathbf{WK}}}
\newcommand\WKt{\WK^{e}}
\makeatletter
\providecommand*{\Dashv}{%
	\mathrel{%
		\mathpalette\@Dashv\vDash
	}%
}
\newcommand*{\@Dashv}[2]{%
	\reflectbox{$\m@th#1#2$}%
}
\makeatother

\newcommand\PWK{\mathsf{PWK}}
\newcommand\PWKe{\mathsf{PWK}_{\mathrm{e}}}
\newcommand\MPWK{\ensuremath{\mathsf{PWK}^{\Box}_{\mathrm{e}}}\xspace}

\newcommand\Ho{\mathrm{H_{0}}}

\newcommand\B{\mathsf{B}_{\mathrm{e}}}

\newcommand{\class}{\mathsf}

\newcommand{\Jzero}{J_{_0}}
\newcommand{\Juno}{J_{_1}}
\newcommand{\Jdue}{J_{_2}}

\newcommand{\Ji}{J_{_k}}

\newcommand\ant{\nicefrac12}
\newcommand\ti{\mathrm{t}}
\newcommand\pair[1]{{\langle#1\rangle}}

\newcommand{\sineq}{\mathrel{\dashv\mkern1.5mu\vdash}}  

\usepackage{hyperref}
\usepackage{amsaddr}

\usepackage{ragged2e}

\title[]{Modal weak Kleene logics: axiomatizations and relational semantics.}
\author{S. Bonzio}\address{Department of Mathematics and Computer Science, \\
	University of Cagliari, Italy.}
\author{N. Zamperlin}
\address{Department of Pedagogy, Psychology and Philosophy, \\
	University of Cagliari, Italy.}

\begin{document}
	
	\maketitle

 \begin{abstract}
Weak Kleene logics are three-valued logics characterized by the presence of an infectious truth-value. In their external versions, as they were originally introduced by Bochvar \cite{Bochvar38} and Halld\'{e}n \cite{Hallden}, these systems are equipped with an additional connective capable of expressing whether a formula is classically true. In this paper we further expand the expressive power of external weak Kleen logics by modalizing them with a unary operator. The addition of an alethic modality gives rise to the two systems $\B^{\square}$ and $\MPWK$, which have two different readings of the modal operator. We provide these logics with a complete and decidable Hilbert-style axiomatization w.r.t. a three-valued possible worlds semantics. The starting point of these calculi are new axiomatizations for the non-modal bases $\B$ and $\PWKe$, which we provide using the recent algebraization results about these two logics. In particular, we prove the algebraizability of $\PWKe$. Finally some standard extensions of the basic modal systems are provided with their completeness results w.r.t. special classes of frames.   
 \end{abstract}
 
\section{Introduction}\label{sec: Intro}

Modal logics are formalisms originally devised to speak about necessity and possibility of formulas belonging to classical propositional, or first-order, logic. They have been successfully extended to non-classical logics, namely modalities can be applied to formulas obeying rules of non-classical propositional (or first-order) logics. The process has involved a large variety of non-classical logics, including (but not limited to) intuitionistic logic \cite{FischerServi},\cite{FisherServi1981},\cite{FontModalityAP}, strong Kleene logic \cite{Fitting},\cite{Fitting2}, Belnap-Dunn \cite{RivieccioModal}, various fuzzy logics \cite{ModalFuzzy}, \cite{LukasiewiczModal}, \cite{LukasiewiczModal2}, \cite{GoedelModal}, \cite{ModalProduct}, and, more in general, the realm of substructural logics \cite{BouModal}. 

This tendency has only marginally involved weak Kleene logics. These (three-valued) logics, originally introduced by Bochvar \cite{Bochvar38}, and subsequently investigated by Kleene \cite{KleeneBook}, to deal with mathematical paradoxes and formulas possibly referring to non-existing objects and/or incorrectly written computer programs, have attracted much attentions in the recent past from several points of view: semantical \cite{CiuniCarrara}, algebraic \cite{Bonzio16SL}, \cite{Bonziobook}, epistemic \cite{SzmucEpistemic}, \cite{BoemBonzio}, computer-theoretic \cite{carrara_computation21}, \cite{ciuni_computererror19}, topic-theoretic \cite{beall2016off}, 
and belief revision \cite{carrara_pwk_argumentation}. To the best of our knowledge, the only existing proposals of modal logics based on (some) weak Kleene logics are due to the works of Correia \cite{Correia} and Segerberg \cite{Segerberg67}. Both papers focus on the modalized extension of Paraconsistent weak Kleene.\footnote{Actually, Correia \cite{Correia} introduces also the idea of a modalized version of Bochvar logic, but gives no axiomatization nor semantical analysis for that.} While Correia's work introduces an axiomatization and a relational semantics for the modal version of PWK, Segerberg's \cite{Segerberg67} is based on the external version of Paraconsistent Weak Kleene and takes in consideration the existing difference between truth and non-falsity, within the three-valued realm, only with reference to establishing the semantical interpretation of modal formulas in a three-valued relational settings, while forgetting that such distinction is already made clear in the choice of the designated truth-values, leading either to Bochvar logic (whose consequence relation preserves truth only) or to Paraconsistent weak Kleene (or, Halld\'{e}n logic, preserving non-falsity). Segerberg \cite{Segerberg67} essentially defines two different necessity operators on the external version of Paraconsistent weak Kleene. In the present work, inspired by Segerberg's intuitions but guided by the above explained distinction, we will introduce modal logics, with a unique necessity operator, on both (the external versions of) Bochvar and Paraconsistent weak Kleene. Remarkably, the work of Segerberg relies on the \emph{external} version of (a) weak Kleene logic. External weak Kleene logics are defined in an extended language from that of (propositional) classical logic, usually adopted also for weak Kleene, which we may now refer to as ``internal Kleene logics''. The founding fathers of these formalisms -- Bochvar and Halld\'{e}n \cite{Bochvar38}, \cite{Hallden} -- actually defined their logics in the external language, which allows for a significant enrichment in expressiveness and the chance of recovering (propositional) classical logic, for a fragment of language, still working within a three-valued semantics. External Kleene logics turn out to be mathematically more interesting (than weak Kleene logics, as usually intended) as they are algebraizable in the sense of Blok and Pigozzi (see e.g. \cite{FontBook}, Definition 3.11). This is part of our motivation to study modal Kleene logics in the external language: although the present work does not take in consideration the global version of modal logics over a certain relational structure, but only their local versions, we still expect to provide a modal basis on which algebraizability can be carried over from the propositional level (if one considers the global consequence relations instead of the local ones). Moreover, as weak Kleene logics turned out to be a particular case of a more general phenomenon, that of the ``logics of variable inclusion'' \cite{Bonziobook} -- ``internal'' modal Kleene logics are indeed examples of logics of variable inclusion. As described in details in \cite{Bonziobook}, these logics could be approached, syntactically, by imposing certain constraints on the inclusion of variables to standard modal logics and, semantically, by considering the construction of the P\l onka sum of modal algebras or of its subvarieties ``corresponding'' to the extension of the modal logic $K$.

The main contribution of the present work is to introduce two different modal logics whose propositional bases are Bochvar external logic and the external version of Paraconsistent weak Kleene logic, respectively. For both of them, we provide Hilbert-style axiomatizations which are sound and complete with respect to a relational semantics, consisting of Kripke frames where the necessity operator is interpreted according to the choices of truth and non-falsity preservation imposed by the propositional basis of each logic. Finally, our work tries to shade a light on some extensions of weak Kleene modal logics, in particular, to those whose semantics is based on reflexive, transitive, and euclidean frames. The aim of this choice is that of providing useful tools for the analysis of epistemic concepts, such as knowledge, beliefs or ignorance, via non-classical logics (a tendency that has already be started e.g. in \cite{Ignoranzasevera},\cite{KnowledgeinBelnap-Dunn}).

The paper is organized into six Sections (including this Introduction) plus and Appendix. In Section \ref{sec: external logics}, we recall several important notions related to external Kleene logics that are needed to go through the paper. We also establish the algebraizability of PWK external logic, and use the algebraizability of both logics to provide new Hilbert-style axiomatizations, which we will argue have some advantages over the existing ones. In Section \ref{sec: Modal}, we introduce, via a Hilbert-style axiomatization, a modal logic based on Bochvar external logic, for which we prove completeness and decidability with respect to the relational semantics. In section \ref{sec: modal PWK}, we proceed analogously for PWK external logic. We dedicate Section \ref{sec: estensioni} to introduce some axiomatic extensions of both modal Bochvar logic and modal PWK and show their completeness with respect to reflexive, transitive and euclidean relational models. Finally, in the Appendix (section \ref{sec: Appendice}), we prove the algebraizability of PWK external logic (with respect to the quasi-variety of Bochvar algebras as equivalent algebraic semantics).

\section{External weak Kleene logics}\label{sec: external logics}

Kleene three-valued logics are traditionally divided into two families, depending on the meaning given to the connectives: \emph{strong Kleene} logics -- counting strong Kleene \cite{KleeneBook} and the logic of paradox \cite{Priestfirst} -- and \emph{weak Kleene} logics, namely Bochvar logic and paraconsistent weak Kleene \cite{Bonzio16SL} (sometimes referred to as Halld\'{e}n's logic \cite{Hallden}). All the mentioned four logics are traditionally intended, and thus defined, over the (algebraic) language of classical logic. However, the intent of the first developer of the \emph{weak Kleene} formalism, D. Bochvar, was to work within an enriched language allowing to express all classical ``two-valued'' formulas -- which he referred to as \emph{external formulas} -- beside the genuinely ``three-valued'' ones. The result of this choice is the following language.

\begin{definition}
    Let us fix the language $\mathcal{L}\colon\langle \neg,\lor,\Jdue,0,1\rangle$ of type $(1,2,1,0,0)$, which is obtained by enriching the classical language by an additional unary connective $\Jdue$ (and the constants $0,1$), where the formula $\Jdue\varphi$ is to be read as ``$\varphi$ is true''. The language $\mathcal{L}$ will be referred to as \emph{external language}, while its $\Jdue$-reduct will be called \emph{internal}. Let us refer to $\Fm$ as the formula algebra over $\mathcal{L}$, and to $Fm$ as its universe. More explicitly, formulas are inductively defined as follows:
    \begin{center}
        $p\in Var$ $|$ $0$ $|$ $1$ $|$ $\neg\varphi$ $|$ $\varphi\lor\psi$ $|$ $\Jdue\varphi.$
    \end{center}
    We will employ the following abbreviations: $\varphi\land\psi\coloneqq\neg(\neg\varphi\lor\neg\psi) $, $\varphi\to\psi\coloneqq\neg\varphi\lor\psi $, $\varphi\leftrightarrow\psi\coloneqq(\varphi\to\psi)\land(\psi\to\varphi) $, $\Jzero\varphi\coloneqq \Jdue\neg\varphi $, $\Juno\varphi\coloneqq \neg(\Jdue\varphi\lor\Jdue\neg\varphi)$, $+\varphi\coloneqq\neg\Juno\varphi$, and $\varphi\equiv\psi\coloneqq\displaystyle(\Jdue\phi\leftrightarrow\Jdue\psi)\land(\Jzero\phi\leftrightarrow\Jzero\psi) $.
\end{definition}
	
The (algebraic) interpretation of the language $\mathcal{L}$ is given via the three-element algebra $\WKt=\pair{\{0,1,\ant\}, \neg,\lor,\Jdue,0,1}$ displayed in Figure \ref{fig:WKe}. 
	\begin{figure}[h]
		\begin{center}\renewcommand{\arraystretch}{1.25}
			\begin{tabular}{>{$}c<{$}|>{$}c<{$}}
				& \lnot  \\[.2ex]
				\hline
				1 & 0 \\
				\ant & \ant \\
				0 & 1 \\
			\end{tabular}
			\qquad
			\begin{tabular}{>{$}c<{$}|>{$}c<{$}>{$}c<{$}>{$}c<{$}}
				\lor & 0 & \ant & 1 \\[.2ex]
				\hline
				0 & 0 & \ant & 1 \\
				\ant & \ant & \ant & \ant \\          
				1 & 1 & \ant & 1
			\end{tabular}
			\qquad
			\begin{tabular}{>{$}c<{$}|>{$}c<{$}}
				\varphi & J_{_2} \varphi \\[.2ex]
				\hline
				1 & 1 \\
				\ant & 0 \\
				0 & 0 \\
			\end{tabular}
		\end{center}
		\caption{The algebra $\WKt$.}\label{fig:WKe}
	\end{figure}

The third value $\ant$ is traditionally read as ``meaningless'' (see e.g. \cite{ferguson2017meaning} and \cite{SmuczMeaningless}) due to its infectious behavior. The defined connectives of conjunction and material implication have the expected tables:

        \begin{figure}[h]
		\begin{center}\renewcommand{\arraystretch}{1.25}
                \begin{tabular}{>{$}c<{$}|>{$}c<{$}>{$}c<{$}>{$}c<{$}}
				\land & 0 & \ant & 1 \\[.2ex]
				\hline
				0 & 0 & \ant & 0 \\
				\ant & \ant & \ant & \ant \\          
				1 & 0 & \ant & 1
			\end{tabular}
			\qquad
                \begin{tabular}{>{$}c<{$}|>{$}c<{$}>{$}c<{$}>{$}c<{$}}
				\to & 0 & \ant & 1 \\[.2ex]
				\hline
				0 & 1 & \ant & 1 \\
				\ant & \ant & \ant & \ant \\          
				1 & 0 & \ant & 1
			\end{tabular}
            \end{center}
        \end{figure} 

Via $\Jdue$ one can define other external connectives: $\Jzero$ expressing the falseness of a formula, and $\Juno$ expressing its non-classicality. Moreover, $\equiv$ is the proper logical equivalence for the external language. Their interpretation in $\WKt$ is displayed in the following tables.\footnote{We are adopting here the notation due to Finn and Grigolia \cite{FinnGrigolia}; Bochvar \cite{Bochvar38} and Segerberg \cite{Segerberg65} instead used $\ti$, $\mathsf{f}$, $-$ for $\Jdue$, $\Jzero$, $\Juno$, respectively.} 
\vspace{5pt}
\begin{center}
\begin{tabular}{>{$}c<{$}|>{$}c<{$}}
				\varphi & \Jzero \varphi \\[.2ex]
				\hline
				1 & 0 \\
				\ant & 0 \\
				0 & 1 \\
			\end{tabular}
			\qquad
			\begin{tabular}{>{$}c<{$}|>{$}c<{$}}
				\varphi & \Juno\varphi  \\[.2ex]
				\hline
				1 & 0 \\
				\ant & 1 \\
				0 & 0 \\
			\end{tabular}
                \qquad
                \begin{tabular}{>{$}c<{$}|>{$}c<{$}>{$}c<{$}>{$}c<{$}}
				\equiv & 0 & \ant & 1 \\[.2ex]
				\hline
				0 & 1 & 0 & 0 \\
				\ant & 0 & 1 & 0 \\          
				1 & 0 & 0 & 1
			\end{tabular}
	\end{center}
\vspace{5pt}

In the interpretation of the language $\mathcal{L}$ some formulas are evaluated into $\{0,1\}$ \emph{only} (which is the universe of a Boolean subalgebra of $\WKt$), by any homomorphism $h\colon\Fm\to\WKt$: these are called \emph{external} formulas (see Definition \ref{def: formula esterna}, for a syntactic definition).  

\subsection{Bochvar external logic}

We recall that a logic $\mathsf{L}$ is induced by a matrix $\pair{\mathbf{M},F}$ (of the same language) when
$$\Gamma\vdash_\mathsf{L}\varphi \text{ iff for all homomorphisms } h\colon\Fm\to\mathbf{M},$$ $$h[\Gamma]\subseteq F \text{ implies } h(\varphi)\in F.$$

\begin{definition}\label{def.: logiche esterne}
	Bochvar \emph{external} logic $\B$ is the logic induced by the matrix $\pair{\WKt,\{1\}}$. PWK \emph{external} logic $\PWKe$ is the logic induced by the matrix $\pair{\WKt,\{1,\ant\}}$.
\end{definition}
Therefore, $\B$ is the logic preserving only the value $1$ (for truth), while $\PWKe$ preserves both $1$ and $\ant$ (thus, non-falsity). The latter, originally introduced by Halld\'{e}n \cite{Hallden}, has been later on studied by Segerberg \cite{Segerberg65}, who named it $\Ho$. Both $\B$ and $\PWKe$ are finitary logics, as they are defined by a finite set of finite matrices. 

Hilbert-style axiomatizations of the two logics are given by Finn-Grigolia \cite{FinnGrigolia} and Segerberg \cite{Segerberg65}, respectively. In order to introduce them, some technicalities are needed.

\begin{definition}\label{def: variabili aperte/coperte}
	An occurrence of a variable $x$  in a formula $\varphi$ is \emph{open} if it does not fall under the scope of $\Ji$, for every $k\in\{0,1,2\}$. A variable $x$ in $\varphi$ is \emph{covered} if all of its occurrences are not open, namely if for every occurrence of $x$ in $\varphi$ falls under the scope of $\Ji$, for some $k\in\{0,1,2\}$.
\end{definition}

The intuition behind the notion of external formula is made precise by the following.

\begin{definition}\label{def: formula esterna}
	A formula $\varphi\in Fm$ is called \emph{external} if all its variables are covered. 
\end{definition}


The following axioms (and rule) define Finn and Grigolia's Hilbert-style axiomatization\footnote{To be precise, in \cite{FinnGrigolia} the following definition of the connective is given: $\varphi\equiv\psi\coloneqq\displaystyle\bigwedge_{i=0}^{2}J_{i}\varphi\leftrightarrow J_{_i}\psi $. Nonetheless, since the operator $\Juno$ is entirely determined by $\Jdue$ and $\Jzero$, Finn and Grigolia's definition can be safely substituted with ours, obtaining an equivalent calculus (see e.g. \cite{SMikBochvar}).} of $\B$ \cite{FinnGrigolia}.

\vspace{5pt}
\noindent
\textbf{Axioms}
\begin{itemize}
	\item[(A1)] $(\varphi\lor\varphi)\equiv\varphi$;
	\item[(A2)] $(\varphi\lor\psi)\equiv(\psi\lor\varphi)$;
	\item[(A3)] $((\varphi\lor\psi)\lor \chi)\equiv(\varphi\lor(\psi\lor\chi))$;
	\item[(A4)] $(\varphi\land(\psi\lor \chi)\equiv((\varphi\land\psi)\lor(\varphi\land\chi))$;
	\item[(A5)] $\neg(\neg \varphi)\equiv\varphi$;
	\item[(A6)] $\neg 1\equiv 0$;
	\item[(A7)] $\neg( \varphi\lor\psi)\equiv(\neg\varphi\land\neg\psi)$;
	\item[(A8)] $0\vee \varphi\equiv\varphi$;
	\item[(A9)] $J_{_2}\alpha\equiv\alpha$;
	\item[(A10)] $J_{_0}\alpha\equiv\neg\alpha$;
	\item[(A11)] $J_{_1}\alpha\equiv 0$;
	\item[(A12)] $J_{_i}\neg\varphi\equiv J_{_2-i}\varphi$, for any $i\in\{0,1,2\}$;
	\item[(A13)] $ J_{_i}\varphi \equiv\neg(J_{_j}\varphi\vee J_{_k}\varphi)$, with $i\neq j\neq k\neq i$;
	
	\item[(A14)] $(J_{_i}\varphi\vee\neg J_{_i}\varphi)\equiv 1$, with $i\in\{0,1,2\}$;
	
	\item[(A15)] $((J_{_i}\varphi\vee J_{_k}\psi)\land J_{_i}\varphi)\equiv J_{_i}\varphi$, with $i,k\in\{0,1,2\}$;
	\item[(A16)] $(\varphi\lor J_{_i}\varphi)\equiv \varphi$, with $i\in\{1,2\}$;
	\item[(A17)] $J_{_0}(\varphi\lor\psi)\equiv J_{_0}\varphi\land J_{_0}\psi$;
	\item[(A18)] $J_{_2}(\varphi\lor\psi)\equiv ( J_{_2}\varphi\land J_{_2}\psi)\vee( J_{_2}\varphi\land J_{_2}\neg\psi)\vee ( J_{_2}\neg \varphi\land J_{_2}\psi) $.
 \end{itemize}

 Let $\alpha,\beta,\gamma$ denote external formulas only:

 \begin{itemize}
        \item[(A19)] $\alpha\to(\beta\to\alpha)$; 
        \item[(A20)] $(\alpha\to(\beta\to\gamma))\to((\alpha\to\beta)\to(\alpha\to\gamma))$;
        \item[(A21)] $(\neg\alpha\to\neg\beta)\to(\beta\to\alpha)$.
\end{itemize}


	
	
	
\vspace{10pt}
\noindent
\textbf{Deductive rule}
\[
\begin{prooftree}
	\Hypo{\varphi}  \Hypo{\varphi\to\psi}
	\Infer[left label = {[MP]}]2 \psi
\end{prooftree}
\]

\vspace{10pt}

Notice that there is nothing special about the choice of axioms (A19)-(A21): it is only important that, together with \emph{modus ponens}, yield a complete axiomatization for classical logic, but only relative to external formulas.\footnote{Actually the original presentation in \cite{FinnGrigolia} contains a much longer axiomatization.}

The fact that $\B$ coincides with the logic induced by the above introduced Hilbert-style axiomatization has been proved in \cite{Ignoranzasevera} (Finn and Grigolia \cite[Theorem 3.4]{FinnGrigolia} only proved a weak completeness theorem for $\B$). We will henceforth indicate by $\vdash_{\B}$ both the consequence relation induced by the matrix $\pair{\WKt,\{1\}}$ and the one induced by the above Hilbert-style axiomatization.

The logic $\B$ is algebraizable with the quasi-variety of Bochvar algebras as its equivalent algebraic semantics.\footnote{This quasi-variety has been introduced by Finn and Grigolia \cite{FinnGrigolia}, while its structural properties are studied in \cite{SMikBochvar}.} This means that there exists maps $\tau\colon Fm\to \mathcal{P}(Eq)$, $\rho\colon Eq\to \mathcal{P}(Fm)$ from formulas to sets of equations and from equations to sets of formulas such that
\[\gamma_{1},\dots,\gamma_{n}\vdash_{\mathsf{B}_{e}}\varphi\iff\tau(\gamma_{1}),\dots,\tau(\gamma_{n})\vDash_{\mathcal{BCA}}\tau(\varphi)\]
and
\[\varphi\thickapprox\psi\Dashv\vDash_{\mathcal{BCA}}\tau\rho(\varphi\thickapprox\psi).\]

The algebraizability of $\B$ is witnessed by the transformers $\tau(\varphi)\coloneqq \{\varphi\thickapprox 1\}$ and  $\rho(\varphi\thickapprox\psi)\coloneqq\{\varphi\equiv\psi\}$ (see \cite{Ignoranzasevera} for details) and allows to provide a ``standard'' Hilbert-style axiomatization, whose axioms and rules make no difference bewteen external and non-external formulas. Recall that for a class $\mathcal{C}$ of algebras (of a certain type), the equational consequence relation $\Theta\vDash_{\mathcal{C}}\phi\approx\psi$ holds iff for all $\A\in\mathcal{C}$ and all homomorphisms (from the formulas in the same type) $h:\Fm\to\A$, if $h(\delta)=h(\epsilon)$ for all $\delta\approx\epsilon\in\Theta$, then $h(\phi)=h(\psi)$.

Using the equational description of the quasi-variety of Bochvar algebras presented in \cite{SMikBochvar} (see in particular \cite[Theorem 7]{SMikBochvar}), we can apply the algorithm described in \cite[Proposition 3.47]{FontBook}, obtaining the following Hilbert-style calculus.

\begin{definition}\label{def: assiomatizzazione B}
A Hilbert-style axiomatization of $\B$ is given by the following Axioms and Rules. \\
\noindent
\textbf{Axioms}
\begin{itemize}
	\item[($\rho$-B1)] $\varphi\lor\varphi\equiv\varphi$;
	\item[($\rho$-B2)] $\varphi\lor\psi\equiv\psi\lor\varphi$;
	\item[($\rho$-B3)] $(\varphi\lor\psi)\lor \chi\equiv\varphi\lor(\psi\lor\chi)$;
	\item[($\rho$-B4)] $\varphi\lor 0\equiv\varphi$;
        \item[($\rho$-B5)] $\neg\neg\varphi\equiv\varphi$;
        \item[($\rho$-B6)] $\neg(\varphi\lor\psi)\equiv\neg\varphi\land\neg\psi$;
        \item[($\rho$-B7)] $\neg 1\equiv 0$;
        \item[($\rho$-B8)] $\varphi\land(\psi\lor \chi)\equiv(\varphi\land\psi)\lor(\varphi\land\chi)$;
	\item[($\rho$-B9)] $\Jzero\Jdue\varphi\leftrightarrow\neg\Jdue\varphi$;
	\item[($\rho$-B10)] $\Jdue\varphi\leftrightarrow\neg(\Jzero\varphi\lor\Juno\varphi)$;
	\item[($\rho$-B11)] $\Jdue\varphi\lor\neg\Jdue\varphi\leftrightarrow 1$;
	\item[($\rho$-B12)] $\Jdue(\varphi\lor\psi)\leftrightarrow(\Jdue\varphi\land\Jdue\psi)\lor(\Jdue\varphi\land\Jzero\psi)\lor(\Jzero\varphi\land\Jdue\psi)$.
\end{itemize}
\vspace{10pt}
\noindent
\textbf{Deductive rules}
\begin{itemize}
        \item[($\rho$-B13)] $\Jdue\varphi\leftrightarrow\Jdue\psi,\Jzero\varphi\leftrightarrow\Jzero\psi\vdash\varphi\equiv\psi$;
        \item[(BAlg3)] $\varphi\dashv\vdash\Jdue\varphi\leftrightarrow 1$.
\end{itemize}

\end{definition}

The axiomatization is equivalent to Finn and Grigolia's calculus, moreover it consists of a proper set of axiom schemata. In fact Finn and Grigolia impose a syntactic restriction on axioms (A19)-(A21), as a result those are not schemata, instead each point represents a countable set of schemata. On the contrary, the above axiomatization makes no distinction between external and non-external formulas, hence it enjoys a proper closure under substitution.

The following results recaps some basic properties of $\B$ which will be used in the following sections.
\begin{lemma}\label{lemma: fatti basilari logica B}
	The following facts hold in $\B$: 
	\begin{enumerate}
            \item $\varphi,\varphi\to\psi\vdash \psi$; 
		\item   $\vdash \alpha\leftrightarrow\Jdue\alpha $, for $\alpha$ external formula; 
            \item $\varphi\sineq\Jdue\varphi $; 
            \item If $\alpha$ is an external formula and a classical theorem then $\vdash\alpha$;
      	\item $\neg\varphi\sineq\Jzero\varphi $;
            \item $\vdash\Jzero\varphi\to\neg\Jdue\varphi $;
            \item $\vdash\Jzero\varphi\land\Jdue\varphi\to 0$;
            \item $\vdash\neg\Jdue\varphi\to\Juno\varphi\lor\Jzero\varphi$;
            \item $\vdash\Juno\varphi\leftrightarrow\Juno\neg\varphi$;
            \item $\vdash\Juno\alpha\to 0$ for $\alpha$ external formula;
            \item $\neg\Jdue\neg\varphi \vdash \Jdue\varphi\lor\Juno\varphi$.
            
  
	\end{enumerate}
 where $\varphi\leftrightarrow\psi$ is an abbrevation for $(\varphi\to\psi)\land(\psi\to\varphi)$.
\end{lemma}

Finally, let us recall that $\B$ has a deduction theorem, in the following form. 

\begin{theorem}[Deduction Theorem for $\B$]\label{th: deduzione per Be}
	It holds that $\Gamma\vdash_{\B}\varphi$ iff there exist some formulas $\gamma_1,\dots,\gamma_n\in \Gamma$ such that $\vdash_{\B}\Jdue\gamma_1\wedge\dots\wedge\Jdue\gamma_n\to \Jdue\varphi$.
\end{theorem}

\subsection{External paraconsistent weak Kleene logic}

The Hilbert-style axiomatization for $\PWKe$, introduced by Segerberg \cite{Segerberg65} is the following.

\vspace{5pt}
\noindent
\textbf{Axioms}
\begin{enumerate}
\item[(A1)] $(\varphi\lor\varphi)\to\varphi$;
\item[(A2)] $\varphi\to(\varphi\lor\psi)$;
\item[(A3)] $(\varphi\lor\psi)\to(\psi\lor\varphi)$;
\item[(A4)] $(\varphi\to\psi)\to((\gamma\lor\varphi)\to(\gamma\lor\psi))$;
\item[(A5)] $(\varphi\land\psi)\to\lnot(\lnot\varphi\lor\lnot\psi)$;
\item[(A6)] $\lnot(\lnot\varphi\lor\lnot\psi)\to(\varphi\land\psi)$;
\item[(A7)] $\varphi\to 1$;
\item[(A8)] $0\to\varphi$;
\item[(A9)] $\varphi\to\Jdue\varphi$;
\item[(A10)] $\Jdue\varphi\to\neg\Jzero\varphi$;
\item[(A11)] $\Jdue(\varphi\land\psi)\leftrightarrow\Jdue\varphi\land\Jdue\psi$;
\item[(A12)] $\Jdue(\varphi\vee\psi)\leftrightarrow((\Jdue\varphi\land\Jdue\psi)\vee(\Jdue\varphi\land\Jzero\psi)\vee(\Jzero\varphi\land\Jdue\psi))$.
\end{enumerate}

\vspace{10pt}
\noindent
\textbf{Deductive rule}
\[
\begin{prooftree}
  \Hypo{\varphi}  \Hypo{\varphi\to\psi}
  \Infer[left label = {[RMP]}]2[\quad provided that no variable is open in $\varphi$ and covered in $\psi$ .] \psi
\end{prooftree}
\]

As before, there is nothing special behind the choice of the axioms (A1) to (A8): one can simply choose any set of axioms which, together with the (usual) rule of Modus Ponens yields an axiomatization of (propositional) classical logic. Observe that the rule [RMP] consists of a linguistic restriction of the standard rule of Modus Ponens: a fact that shall not surprise, as a very similar restricted rule has been introduced for an axiomatization of $\PWK$ (in the language of classical logic) \cite{Bonzio16SL}. However, providing the same logic with a ``standard'' calculus presenting no linguistic restriction is preferable (for instance, for internal $\PWK$, such Hilbert-style axiomatizations can be found in \cite{MarcelinoHilbertPWK}, \cite{BonzioMorascoPrabaldi}). As in the case of Bochvar, also for $\PWKe$, we will indicate by $\vdash_{\PWKe}$ both the consequence relation induced by the matrix $\pair{\WKt,\{1,\ant\}}$ and the one induced by the above Hilbert-style axiomatization.

Like $\B$, also the logic $\PWKe$ is algebraizable with quasi-variety of Bochvar algebras as its equivalent algebraic semantic. The transformers that witness the algebraizability are $\tau(\varphi)\coloneqq \{\neg\Jzero\varphi\thickapprox 1\}$ and  $\rho(\varphi\thickapprox\psi)\coloneqq\{\varphi\equiv\psi\}$. We leave the proof of this result in the Appendix (see Section \ref{sec: Appendice}). Observe that, although it is not very common (see \cite[pp.121-122]{FontBook}), the same class of algebras can play the role of equivalent algebraic semantics for different logics (clearly, the algebraizability is given by different transformers).

This algebraizability allows us to apply the algorithm of \cite{FontBook} already employed for $\B$, obtaining a Hilbert-style axiomatization alternative to Segerberg's.

\begin{definition}\label{def: assiomatizzazione PWK}
A Hilbert-style axiomatization of $\PWKe$ is given by the following Axioms and Rules. \\
\noindent
\textbf{Axioms}
\begin{itemize}
	\item[($\rho$-B1)] $\varphi\lor\varphi\equiv\varphi$;
	\item[($\rho$-B2)] $\varphi\lor\psi\equiv\psi\lor\varphi$;
	\item[($\rho$-B3)] $(\varphi\lor\psi)\lor \chi\equiv\varphi\lor(\psi\lor\chi)$;
	\item[($\rho$-B4)] $\varphi\lor 0\equiv\varphi$;
        \item[($\rho$-B5)] $\neg\neg\varphi\equiv\varphi$;
        \item[($\rho$-B6)] $\neg(\varphi\lor\psi)\equiv\neg\varphi\land\neg\psi$;
        \item[($\rho$-B7)] $\neg 1\equiv 0$;
        \item[($\rho$-B8)] $\varphi\land(\psi\lor \chi)\equiv(\varphi\land\psi)\lor(\varphi\land\chi)$;
	\item[($\rho$-B9)] $\Jzero\Jdue\varphi\leftrightarrow\neg\Jdue\varphi$;
	\item[($\rho$-B10)] $\Jdue\varphi\leftrightarrow\neg(\Jzero\varphi\lor\Juno\varphi)$;
	\item[($\rho$-B11)] $\Jdue\varphi\lor\neg\Jdue\varphi\leftrightarrow 1$;
	\item[($\rho$-B12)] $\Jdue(\varphi\lor\psi)\leftrightarrow(\Jdue\varphi\land\Jdue\psi)\lor(\Jdue\varphi\land\Jzero\psi)\lor(\Jzero\varphi\land\Jdue\psi)$.
\end{itemize}
\vspace{10pt}
\noindent
\textbf{Deductive rules}
\begin{itemize}
        \item[($\rho$-B13)] $\Jdue\varphi\leftrightarrow\Jdue\psi,\Jzero\varphi\leftrightarrow\Jzero\psi\vdash\varphi\equiv\psi$;
        \item[(PWKAlg3)] $\varphi\dashv\vdash\neg\Jzero\varphi\leftrightarrow 1$.
\end{itemize}

\end{definition}

The only difference between the new axiomatizations proposed for $\B$ and $\PWKe$ relies on the rules (PWKAlg3) and (BAlg3), which is expected since the two logics have the same equivalent algebraic semantics and differ only for the $\tau$ transformer.

\begin{lemma}\label{lemma: fatti basilare PWKe}

The following facts hold for the logic $\PWKe$: 
\begin{enumerate}

\item $\alpha,\alpha\to\beta\vdash\beta$, for every $\alpha,\beta$ external formulas;
\item $\varphi\sineq \neg\Jzero\varphi $;
\item every theorem of classical logic is a theorem of $\PWKe$;
\item $\vdash\neg\Jzero 0\to 0$;
\item $\vdash \neg\varphi\to\Jzero\varphi$;
\item   $\vdash\alpha\leftrightarrow\Jdue\alpha$, for $\alpha$ external formula;
\item $\varphi\vdash\Jdue\varphi\lor\Juno\varphi$;
\item $\vdash\Jdue\neg\varphi\leftrightarrow\Jzero\varphi$;
\item $\vdash\Juno\varphi\leftrightarrow\Juno\neg\varphi$;
\item $\vdash\neg\Jzero 1$;
\item $\vdash\varphi\to\Jdue\varphi$;
\item $\vdash\Jdue\varphi\to +\varphi$.

\end{enumerate}
\end{lemma}

$\PWKe$ has a Deduction Theorem very similar to $\B$, by just adapting the statement of Theorem \ref{th: teo. deduzione Bochvar modale} (from truth) to non-falsity, in the obvious way suggested by external connectives.

\begin{theorem}[Deduction Theorem for $\PWKe$]\label{lemma: deduzione PWK modale}
It holds that $\Gamma\vdash_{\PWKe}\varphi$ iff there exist some formulas $\gamma_1,\dots,\gamma_n\in \Gamma$ such that $\vdash_{\PWKe}\neg\Jzero\gamma_1\wedge\dots\wedge\neg\Jzero\gamma_n\to \neg\Jzero\varphi$.
\end{theorem}

\section{Modal Bochvar logic}\label{sec: Modal}

From now on, by $\Fm $ we will denote the formula algebra constructed over a denumerable set of propositional variables $Var$ in the language $\mathcal{L}: \neg,\vee, \Jdue, \square ,0,1$ of type $\pair{1,2,1,1,0,0}$. The connectives $\land,\to$ are defined as usual, while recall that $\Jzero\varphi$ and $\Juno\varphi$ are abbreviations for $\Jdue\neg\varphi$ and $ \neg (\Jdue\varphi \lor \Jdue\neg\varphi)$, respectively. Let, moreover, $\diamond\varphi$ be an abbreviation for $\neg\square\neg\varphi$. Our aim with the above introduced language is to define a (local) modal logic whose propositional basis is Bochvar external logic and whose interpretation of formulas $\square\varphi$, in a relational semantics, is that $\square\varphi$ holds in a state when $\varphi$ holds (is equal to 1) in all related states.

We introduce the logic $\pair{\Fm,\vdash_{\B^{\square}}}$ as induced by the following Hilbert-style axiomatization. \\


\noindent
\textbf{Axioms}
\begin{itemize}
	\item the axioms of $\B$ in Definition \ref{def: assiomatizzazione B}; 
	\vspace{3pt}
	\item[(B1)]  $\square(J_2\varphi\to J_2\psi)\to (\square J_2\varphi\to\square J_2\psi)$;
	\vspace{3pt}
	\item[(B2)]  $+\varphi\leftrightarrow +\square \varphi$; 
	\vspace{3pt}
	\item[(B3)] $\Jdue\square\varphi\to\square\Jdue \varphi$;
	\vspace{3pt}
	\item[(B4)] $\Jzero\square\varphi\to\neg\square\Jzero\neg \varphi$.
	
\end{itemize}

\vspace{4pt}
\noindent
\textbf{Deductive rules}

\begin{itemize}
 \item ($\rho$-B13) $\Jdue\varphi\leftrightarrow\Jdue\psi,\Jzero\varphi\leftrightarrow\Jzero\psi\vdash\varphi\equiv\psi$;
 \vspace{5pt}
        \item(BAlg3) $\varphi\dashv\vdash\Jdue\varphi\leftrightarrow 1$;
	\vspace{5pt}
	
	\item (N): if $\vdash \varphi $ then $\vdash \square\varphi$. 
	
	

\end{itemize}

\vspace{10pt}
\noindent
Throughout this Section, for ease of notation, we will write $\vdash$ instead of $\vdash_{\B^{\square}}$ and refer to the above introduced logic simply as $\B^{\square}$. The axiom (B1) can be generalized to all external formulas, as follows. 

\begin{lemma}\label{lemma: K per Bochvar su esterne}
    For $\alpha,\beta$ external formulas, the following is a theorem of $\B^{\square}$:
    \begin{itemize}
        \item[\emph{(BK)}] $\square(\alpha\to\beta)\to (\square\alpha\to\square\beta)$
    \end{itemize}    
\end{lemma}
\begin{proof}
    Consider arbitrary external formulas $\alpha,\beta$. Let us start by instantiating (B1) as $\vdash\square(\Jdue\alpha\to\Jdue\beta)\to (\square\Jdue\alpha\to\square\Jdue\beta)$. Recalling Lemma \ref{lemma: fatti basilari logica B}, for $\gamma$ external it holds $\vdash\gamma\leftrightarrow\Jdue\gamma$, therefore we can substitute equivalent formulas and obtain $\vdash\square(\alpha\to\beta)\to (\square\alpha\to\square\beta)$.
\end{proof}

\begin{remark}\label{rem: MP per Bochvar Modale}
The rule of Modus Ponens (MP) obviously holds for $\B^{\Box}$ (see Lemma \ref{lemma: fatti basilari logica B}).   
\end{remark}


\begin{lemma}\label{lemma: fatti basilari logica modale}
	In the logic $\B^{\square}$ the following facts hold: 
	\begin{enumerate}
		\item $\Jdue\square\varphi\vdash \Jdue\square\Jdue\varphi$; 
		\item $\square(J_i\varphi\land J_k\psi) \leftrightarrow \square J_i\varphi\land\square J_k\psi$ for every $i,k \in \{ 0,1,2 \}$;
		\item $\diamond(J_i\varphi\vee J_k\psi) \leftrightarrow \diamond J_i\varphi\vee\diamond J_k\psi$ for $i,k \in \{ 0,1,2 \}$; 
		\item $\vdash \Jdue\diamond\varphi\to \diamond\Jdue\varphi\lor\diamond\Juno\varphi$.
	\end{enumerate}
\end{lemma}
\begin{proof} Since $\B^{\square}$ contains all axioms from $\B$, we will freely make use of theorems and rules holding in the latter logic. In particular, notice that classical logic can always be employed on external formulas (by Lemma \ref{lemma: fatti basilari logica B} and the fact that MP is a rule of $\B$).
	\begin{enumerate}
		\item By (B3) and (MP), $\Jdue\square\varphi\vdash\square\Jdue \varphi$; by Lemma \ref{lemma: fatti basilari logica B} (and the transitivity of $\vdash$) we get $\square\Jdue\varphi\vdash\Jdue\square\Jdue \varphi$. 
		\item By classical logic, $\vdash J_i\varphi\land J_k \psi \to J_i\varphi$, by (N) $\vdash \square (J_i\varphi\land J_k\psi \to J_i\varphi)$. Applying (BK) and (MP), $\vdash \square (J_i\varphi\land J_k\psi) \to \square J_i\varphi$. By the same reasoning, $\vdash \square (J_i\varphi\land J_k\psi) \to \square J_k\psi$ as well. We conclude $\vdash \square (J_i\varphi\land J_k\psi) \to \square J_i\varphi\land\square J_k\psi$. For the other direction: $\vdash J_i\varphi \to (J_i\psi \to J_i\varphi\land J_k\psi)$ by classical logic. By (N), $\vdash \square(J_i\varphi \to (J_k\psi \to J_i\varphi\land J_k\psi))$, and by (BK) and (MP) twice, we have $\vdash \square J_i\varphi \to (\square J_k\psi \to \square(J_i\varphi\land J_k\psi))$. Again by classical logic $\vdash \square J_i\varphi \land \square J_k\psi \to \square(J_i\varphi\land J_k\psi)$.	
		\item By classical logic, $\vdash \neg J_i\varphi\land\neg J_k\psi \to \neg J_i\varphi$, by (N) $\vdash \square (\neg J_i\varphi\land\neg J_k\psi \to \neg J_i\varphi)$. Applying (BK) and (MP), $\vdash \square (\neg J_i\varphi\land\neg J_k\psi) \to \square\neg J_i\varphi$. By contraposition and de Morgan, $\vdash \neg\square\neg J_i\varphi \to \neg\square\neg (J_i\varphi\lor J_k\psi)$, which is the definition of $\vdash \diamond J_i\varphi \to \diamond (J_i\varphi\lor J_k\psi)$. By the same reasoning, $\vdash \diamond J_k\psi \to \diamond (J_i\varphi\lor J_k\psi)$ as well. We conclude $\vdash \diamond J_i\varphi\lor\diamond J_k\psi \to \diamond (J_i\varphi\lor J_k\psi)$.
		For the other direction, $\vdash \neg J_i\varphi\land\neg J_k\psi \to \neg(J_i\varphi\lor J_k\psi)$ by classical logic. By (N), $\vdash \square(\neg J_i\varphi\land\neg J_k\psi \to \neg(J_i\varphi\lor J_k\psi))$, and by (BK) and (MP) twice we have $\vdash \square(\neg J_i\varphi\land\neg J_k\psi) \to \square\neg(J_i\varphi\lor J_k\psi)$. By point (2) we can distribute box, $\vdash (\square\neg J_i\varphi\land\square\neg J_k\psi) \to \square\neg(J_i\varphi\lor J_k\psi)$. By classical logic, $\vdash \neg(\neg\square\neg J_i\varphi\lor\neg\square\neg J_k\psi) \to \square\neg(J_i\varphi\lor J_k\psi)$. By contraposition we conclude $\vdash \diamond(J_i\varphi\lor J_k\psi) \to \diamond J_i\varphi\lor\diamond J_k\psi$.
		
		\item By (B4) $\vdash \Jzero\Box\neg\varphi\to\neg\Box\Jzero\varphi$. For the linguistic abbreviations introduced, we have that the antecedent $\Jzero\Box\neg\varphi = \Jdue\neg\Box\neg\varphi = \Jdue\diamond\varphi$; while the consequent $\neg\Box\Jzero\varphi = \neg\Box\neg (\Jdue\varphi\lor\Juno\varphi) = \diamond(\Jdue\varphi\lor\Juno\varphi) = \diamond\Jdue\varphi\lor\diamond\Juno\varphi$.
	\end{enumerate}
\end{proof}

\begin{theorem}[Deduction Theorem]\label{th: teo. deduzione Bochvar modale}
	For the logic $\B^{\square}$, it holds that $\Gamma\vdash\varphi$ iff there exist some formulas $\gamma_1,\dots,\gamma_n\in \Gamma$ such that $\vdash\Jdue\gamma_1\wedge\dots\wedge\Jdue\gamma_n\to \Jdue\varphi$.
\end{theorem}
\begin{proof}
	The right to left direction is obvious. The other direction is proved by induction on the length of the derivation of $\varphi$ from $\Gamma$. We just show the inductive case of the rule (N). Let $\varphi=\Box\psi$, for some $\psi\in Fm$, and $\Gamma\vdash\Box\psi$, and the last deduction rule applied is (N), hence it holds $\vdash\psi$. By the latter fact, we have $\vdash\Box\psi$, hence $\vdash\Jdue\Box\psi$ (by Lemma \ref{lemma: fatti basilari logica B}). By induction hypothesis, there exists some formulas $\gamma_1,\dots,\gamma_n\in \Gamma$ such that $\vdash\Jdue\gamma_1\wedge\dots\wedge\Jdue\gamma_n\to \Jdue\psi$. Since the axioms (and rule) of classical logic hold for external formulas (Lemma \ref{lemma: fatti basilari logica B}), we have $\vdash\Jdue\Box\psi\to(\Jdue\gamma_1\wedge\dots\wedge\Jdue\gamma_n\to \Jdue\Box\psi)$, hence, by (MP), $\vdash\Jdue\gamma_1\wedge\dots\wedge\Jdue\gamma_n\to \Jdue\Box\psi$.
\end{proof}

\subsection{Semantics}

The intended semantics of this modal logic consists of a relational (Kripke-style) structure where formulas, in each world, are evaluated into $\WKt$ (this has been already implemented for instance in \cite{Ignoranzasevera}). We introduce these structures according to the current terminology adopted in many-valued modal logics. 

\begin{definition}\label{def: Kripke model}
	A \emph{weak three-valued Kripke model} $\mathcal{M}$ is a structure $\pair{W,R, v}$ such that: 
	\begin{enumerate}
		\item $W$ is a non-empty set (of possible worlds); 
		\item $R$ is a binary relation over $W$ ($R\subseteq W\times W$);
		\item $v$ is a map, called \emph{valuation}, assigning to each world and each variable, an element in $\WKt$ ($v\colon W\times \Fm\to\WKt$). 
	\end{enumerate}
\end{definition}

Non-modal formulas will be interpreted as in $\B$, i.e. we assume that $v$ is a homomorphism, in its second component, with respect to $\neg,\vee,\Jdue,1,0$. The reduct $\mathcal{F}=\pair{W,R}$ of a model $\mathcal{M}$ is called \emph{frame}. \\

\noindent
\textbf{Notation:} for ordered pairs of related elements, we equivalently write $(w,s)\in R$ or $wRs$. \\

The semantical interpretation of the modality $\Box$ is what characterize a special family of weak three-valued Kripke models:

\begin{definition}\label{def: semantica Box}
	A \emph{Bochvar-Kripke model} is a weak three-valued Kripke model $\pair{W,R, v}$ such that $v$ evaluates formulas of the form $\Box\varphi$ according to:
	\begin{itemize}
		\item[(1)] $v(w,\Box\varphi)= 1$ iff $v(w,\varphi)\neq\ant$ and $v(s,\varphi)= 1$ for every $s\in W$ such that $wRs$. 
		\item[(0)] $v(w,\Box\varphi)= 0$ iff $v(w,\varphi)\neq\ant$ and there exists $s\in W$ such that $wRs$ and $v(s,\varphi)\neq 1$. 
		\item[($\ant$)] $v(w,\Box\varphi)= \ant$ iff $v(w,\varphi)=\ant$. 
	\end{itemize}
\end{definition}

To simplify, within this Section by \emph{model} we intend \emph{Bochvar-Kripke model}.

\begin{remark}\label{rem: semantica Diamond}
	Observe that, in every model $\pair{W,R, v}$ with $w\in W$, it holds that $v(w,\diamond\varphi)= 1$ iff $v(w,\varphi)\neq\ant$ and there exists $s\in W$ such that $wRs$ and $v(s,\varphi)\neq 0$, while $v(w,\diamond\varphi) = 0$ iff $v(w,\varphi)\neq\ant$ and $v(s,\varphi)= 0$ for every $s\in W$ such that $wRs$.
\end{remark}

As usual in modal logic, one can opt to study the local or the global consequence relation related to a class of frame. In this paper, we will always deal with the former. Accordingly, we denote by $\models^{l}_{\B^{\square}}$ the \emph{local} modal Bochvar external logic on the class of all frames obtained by taking $\{ 1 \}$ as designated value, that is:

\begin{definition}\label{def: conseguenza logica Bochvar}
        $\Gamma \vDash^{l}_{\B^{\square}} \varphi$ iff for all models $\pair{W,R, v}$ and all $w\in W$, if $v(w,\gamma)=1, \forall\gamma\in\Gamma$, then $v(w,\varphi)=1$.
\end{definition}

In the following we omit the subscript and write simply $\vDash$, instead of $\vDash^{l}$. The following semantic notions are standard.

\begin{definition}\label{def: soddisfacibilitÃ  e validitÃ   Bochvar}
        A formula $\varphi$ is satisfied (valid) in a model $\pair{W,R, v}$ if $v(w,\varphi)=1$, for some (all) $w\in W$. A formula $\varphi$ is valid in a frame $\mathcal{F}$ (notation $\mathcal{F}\vDash \varphi$) if it is valid in $\pair{\mathcal{F},v}$, for all valuations $v$. A formula $\varphi$ is valid in a class of frames $\mathcal{K}$ (notation $\mathcal{K}\vDash \varphi$) if it is valid in every frame $\mathcal{F}\in\mathcal{K}$.
\end{definition}

\subsection{Completeness and decidability}

\begin{definition}\label{def: insieme consistente}
	A set $\Gamma\subset Fm$ is \emph{consistent} if $\Gamma\not\vdash\varphi$, for some $\varphi\in Fm$. It is inconsistent if it is not consistent.
\end{definition}

\begin{remark}\label{rem: caratterizzazione equivalente inconsistenza}
	Equivalently, a set $\Gamma\subset Fm$ is consistent if there is no formula $\varphi\in Fm$, such that $\Gamma\vdash\Jdue\varphi$ and $\Gamma\vdash\neg\Jdue\varphi$. Observe that this is equivalent to say that $\Gamma\not\vdash 0$.  
\end{remark}

\begin{definition}\label{def: insieme massimale consistente}
	A consistent set $\Gamma$ is \emph{maximally consistent} (or \emph{complete}) whenever $\Gamma\subset \Gamma'$ implies that $\Gamma'$ is inconsistent. Equivalently, $\Gamma$ is maximally consistent iff, for every $\varphi\in Fm$ exactly one of the following holds: 
	\begin{itemize}
		\item[i)] $\varphi\in\Gamma$; 
		\item[ii)] $\neg\varphi\in\Gamma$;
		\item[iii)] $\Juno\varphi\in\Gamma$;
	\end{itemize}
\end{definition}

\begin{definition}\label{def: formula sensata in un mondo}
	A formula $\varphi$ is \emph{meaningful} in a maximally consistent set $w$ if $x\in w$ or $\neg x\in w$, for every open variable $x\in \varphi$.
\end{definition}

Observe that the definition of meaningful formulas implies, semantically, that such formulas are those evaluated, in a state, into the two-elements Boolean algebra $\mathbf{B}_2$ only.
The definition of meaningful formula obviously apply to variables as well. 

\begin{lemma}\label{lemma: proprieta mondi massimali consistenti - 1}
	Let $w$ be a maximally consistent set of formulas, then:
	\begin{enumerate}
		\item if $\neg\varphi\not\in w$ and all the variables occurring in $\varphi$ are meaningful in $w$, then $\varphi\in w$; 
		\item if all the variables of $\varphi$ are covered, then $\varphi$ is meaningful in $w$;
		\item if $\vdash\varphi$ then $\varphi\in w$.
	\end{enumerate}
\end{lemma}
\begin{proof}
	We just show (3) (as the other claims can be found also in \cite[Lemma 4.6]{Segerberg67}). Suppose that $\vdash\varphi$ and, by contradiction, that either $\neg\varphi\in w$ or $\Juno\varphi\in w$. Let us assume that $\neg\varphi\in w$. From $\vdash\varphi$ it follows $\varphi\in w$. By Lemma \ref{lemma: fatti basilari logica B} we have $w\vdash\Jdue\varphi$ and $w\vdash\Jzero\varphi$. Applying the same lemma, the latter yields $w\vdash\neg\Jdue\varphi$, in contradiction with the assumption that $w$ is (maximally) consistent (see Remark \ref{rem: caratterizzazione equivalente inconsistenza}). One can reason similarly for the case $\Juno\varphi\in w$.
\end{proof}

\begin{lemma}\cite[Lemma 4.7]{Segerberg67}\label{lemma: sulle formule sensate nei mondi massimali}
	Let $w$ be a maximally consistent set of formulas, t.f.a.e.
	\begin{enumerate}
		\item $\varphi$ is meaningful in $w$; 
		\item either $\varphi\in w$ or $\neg\varphi\in w$; 
		\item $+\varphi\in w$; 
		\item $\Box\varphi$ is meaningful in $w$;
		\item $\diamond\varphi$ is meaningful in $w$.
	\end{enumerate}
\end{lemma}

\begin{lemma}\cite[Lemma 4.8]{Segerberg67}\label{lemma: proprieta mondi massimali consistenti - 2}
	For every maximally consistent set $w$ the following hold:  
	\begin{enumerate}
		\item If $\varphi\to\psi\in w$ and $\varphi\in w$ then $\psi\in w$;
		\vspace{3pt}
		\item $\varphi\land\psi\in w$ if and only if $\varphi,\psi\in w$;
		\vspace{3pt}
		\item $\varphi\lor\psi\in w$ if and only if $\varphi\in w$ or $\psi\in w$;
		\vspace{3pt}
		\item $\varphi\in w$ if and only if $\Jdue\varphi\in w$; 
		\vspace{3pt}
		\item $\Jdue\varphi\in w$ if and only if $\neg\Jdue\varphi\not\in w$.
	\end{enumerate}
\end{lemma}


\begin{lemma}\label{lemma: Insieme inconsistente}
	Let $\Gamma$ be a consistent set of formulas. $\Gamma\cup\{\varphi\}$ is inconsistent if and only if $\Gamma\vdash\neg\varphi$ or $\Gamma\vdash\Juno\varphi$.
\end{lemma}
\begin{proof}
Let $\Gamma$ be a consistent set of formulas.\\
\noindent
($\Rightarrow$) Let $\Gamma\cup\{\varphi\}$ be inconsistent and $\Gamma\not\vdash\neg\varphi$. By assumption, $\Gamma\cup\{\varphi\}\vdash 0$. By Theorem \ref{th: teo. deduzione Bochvar modale}, there exist formulas $\gamma_{1},\dots,\gamma_{n}\in\Sigma$ such that $\vdash\Jdue\gamma_{1}\wedge\dots\wedge\Jdue\gamma_{n}\wedge\Jdue\varphi\to\Jdue 0$, hence $\vdash\Jdue\gamma_{1}\wedge\dots\wedge\Jdue\gamma_{n}\to\neg\Jdue\varphi$, thus $\Gamma\vdash\neg\Jdue\varphi$. Since $\vdash\neg\Jdue\varphi\to\Juno\varphi\lor\Jzero\varphi$ by Lemma \ref{lemma: fatti basilari logica B}, $\Gamma\vdash\Juno\varphi\lor\Jzero\varphi$. By assumption, $\Gamma\not\vdash\neg\varphi$ which implies $\Gamma\not\vdash\Jzero\varphi$, hence $\Gamma\vdash\Juno\varphi$.\\
($\Leftarrow$) Let $\Gamma\vdash\neg\varphi$ or $\Gamma\vdash\Juno\varphi$. Suppose $\Gamma\vdash\neg\varphi$ is the case, hence $\Gamma\vdash\Jzero\varphi$ (by Lemma \ref{lemma: fatti basilari logica B}). On the other hand, $\Gamma\cup\{\varphi\}\vdash\Jdue\varphi$, hence $\Gamma\cup\{\varphi\}\vdash\Jzero\varphi\land\Jdue\varphi$, and since by Lemma \ref{lemma: fatti basilari logica B} it is a theorem that $\vdash\Jzero\varphi\land\Jdue\varphi\to 0$, then $\Gamma\cup\{ \varphi \}$ is inconsistent. The proof is analogue in case $\Gamma\vdash\Juno\varphi$. 
\end{proof}

\begin{lemma}[Lindenbaum's Lemma]\label{lemma: Lindenbaum}
	Let $\Gamma$ be a consistent set of formulas such that $\Gamma\not\vdash\varphi$, for some $\varphi\in Fm$, then there exists a maximally consistent set of formulas $w$ such that $\Gamma\subseteq w$ and such that $\varphi\not\in w$.
\end{lemma}
\begin{proof}
Consider an enumeration $\psi_{1},\psi_{2},\psi_{3},\dots$ of the formulas in $Fm$. Define: 
$$
\Gamma_{0} = \begin{cases}
\Gamma \cup\{\neg\varphi\} \text{ if consistent,}\\
\Gamma \cup\{\Juno\varphi\} \text{ otherwise.}
\end{cases}
$$

$$
\Gamma_{i+1} = \begin{cases}
\Gamma_{i} \cup\{\psi_{i}\} \text{ if consistent, else}\\
\Gamma_{i} \cup\{\neg\psi_{i}\} \text{ if consistent, else}\\
\Gamma _{i}\cup\{\Juno\psi_{i}\}.
\end{cases}
$$
$$w=\bigcup_{i\in\mathbb{N}}\Gamma_{i}.$$

\noindent
Observe that, by construction, $w$ is maximal. We want to show that $w$ is also consistent. We first claim that $\Gamma_{0} $ is consistent. If $\Gamma_{0} = \Gamma \cup\{\neg\varphi\} $ then it is consistent by construction. Differently, $\Gamma_{0} = \Gamma \cup\{\Juno\varphi\} $, which means that $\Gamma \cup\{\neg\varphi\} $ is inconsistent. Hence, by Lemma \ref{lemma: Insieme inconsistente}, $\Gamma\vdash\neg\neg\varphi$ or $\Gamma\vdash\Juno\neg\varphi$. However, $\Gamma\not\vdash\neg\neg\varphi$ (since, by assumption, $\Gamma\not\vdash\varphi$ and $\vdash_{\B}\varphi\leftrightarrow\neg\neg\varphi$), so $\Gamma\vdash\Juno\neg\varphi$, which implies $\Gamma\vdash\Juno\varphi$ (as $\vdash_{\B}\Juno\varphi\leftrightarrow\Juno\neg\varphi$). By Lemma \ref{lemma: Insieme inconsistente}, $\Gamma_{0} = \Gamma \cup\{\Juno\varphi\} $ is consistent if and only if $\Gamma\not\vdash\neg\Juno\varphi$ and 
$\Gamma\not\vdash\Juno\Juno\varphi$. Now, since $\Gamma$ is consistent and $\Gamma\vdash\Juno\varphi$, then $\Gamma\not\vdash\neg\Juno\varphi$. Moreover, since $\Gamma$ is consistent $\Gamma\not\vdash\Juno\Juno\varphi$ (as $\vdash\Juno\Juno\varphi\to 0$). This shows that $\Gamma_{0}$ is consistent. \\
\noindent
We claim that $\Gamma_{i+1}$ is consistent, given that $\Gamma_{i}$ is. So, suppose that $\Gamma_{i}\cup\{\varphi\}$ and $\Gamma_{i}\cup\{\neg\varphi\}$ are inconsistent. Then, by Lemma \ref{lemma: Insieme inconsistente}, $\Gamma_{i}\vdash\neg\varphi$ or $\Gamma_{i}\vdash\Juno\varphi$, and, $\Gamma_{i}\vdash\neg\neg\varphi$ or $\Gamma_{i}\vdash\Juno\neg\varphi$. By consistency of $\Gamma_{i}$, the only possible case is that $\Gamma_{i}\vdash\Juno\varphi$ and $\Gamma_{i}\vdash\Juno\neg\varphi$, from which follows the consistency of $\Gamma\cup\{\Juno\varphi\}$ (indeed, if it is not consistent then, by Lemma \ref{lemma: Insieme inconsistente}, $\Gamma_{i}\vdash\neg\Juno\varphi$, in contradiction with the consistency of $\Gamma_{i}$). This shows that $w$ is maximal and consistent and, by construction, $\neg\varphi\in w$ or $\Juno\varphi\in w$, therefore $\varphi\not\in w$.	
\end{proof}

As a first step to introduce canonical models, let us define the \emph{canonical relation}. 

\begin{definition}\label{def: relazione canonica}
	Let $\mathcal{W}$ be the set of all maximally consistent set of formulas. Then the \emph{canonical relation} $\mathcal{R}\subset \mathcal{W}\times \mathcal{W}$ for $\B^\square$ is defined, for every $w,s\in \mathcal{W}$ as: 
	$$w\mathcal{R} s \text{ iff }\forall\varphi\in Fm \text{ s.t. } \Box\varphi\in w \text{ then } \varphi\in s. $$
	
\end{definition}

\begin{lemma}[Existence Lemma]\label{lemma: di esistenza}
	For every maximally consistent set of formulas $w\in\mathcal{W}$ such that $\diamond\varphi\in w$ (for some $\varphi\in Fm$) then $\varphi$ is meaningful in $w$ and there exists a maximally consistent set $s\in\mathcal{W}$ such that $w\mathcal{R}s$ and either $\varphi\in s$ or $\Juno\varphi\in s$. 
\end{lemma}
\begin{proof}
	Suppose that $\diamond\varphi\in w$, for some $w$ maximally consistent set of formulas. Consider the set
	$$ s^{-} = \{\Jdue\psi | \Box\Jdue\psi\in w\}. $$
	Observe that $s^{-}\neq\emptyset$, as for every formula $\psi$ such that $\vdash\psi$, then $\vdash\Jdue\psi$, hence $\vdash\Box\Jdue\psi$, which implies $\Box\Jdue\psi\in w$, by Lemma \ref{lemma: proprieta mondi massimali consistenti - 1}-(3). Let us show that $ s^{-}$ is consistent. Suppose, by contradiction, that $ s^{-}$ is inconsistent, then $s^{-}\vdash\gamma$, for every $\gamma\in Fm$, thus, in particular, $ s^{-}\vdash\neg\varphi$. By Deduction Theorem there are formulas $\Jdue\psi_1,\dots,\Jdue\psi_n\in s^{-}  $ such that $\vdash \Jdue\Jdue\psi_1\land\dots\land\Jdue\Jdue\psi_n \to \Jdue\neg\varphi$. Recall that $\vdash_{\B}\Jdue\Jdue\gamma\leftrightarrow\Jdue\gamma$, for every $\gamma\in Fm$, thus $\vdash \Jdue\psi_1\land\dots\land\Jdue\psi_n \to \Jdue\neg\varphi$. By applying (N), we get $\vdash \Box(\Jdue\psi_1\land\dots\land\Jdue\psi_n \to \Jdue\neg\varphi)$ and by distributing box ((BK) and Lemma \ref{lemma: fatti basilari logica modale}), $\vdash \Box\Jdue\psi_1\land\dots\land\Box\Jdue\psi_n \to \Box\Jdue\neg\varphi)$. Observe that, by construction of $s^{-}$, $\Box\Jdue\psi_i\in w$, for every $i\in\{1,\dots, n\}$, hence $\Box\Jdue\neg\varphi\in w$, i.e. $\neg\diamond\neg\Jdue\neg\varphi\in w$. By Lemma \ref{lemma: fatti basilari logica B}, $\vdash\neg\Jdue\neg\varphi\leftrightarrow\Jdue\varphi\lor\Juno\varphi$, thus $\neg\diamond(\Jdue\varphi\lor\Juno\varphi)\in w$, which implies (by distributivity of diamond, Lemma \ref{lemma: fatti basilari logica modale}), $\neg(\diamond\Jdue\varphi\lor\diamond\Juno\varphi)\in w$. On the other hand, $\diamond\varphi\in w$, hence $\Jdue\diamond\varphi\in w$, which implies $\diamond\Jdue\varphi\lor\diamond\Juno\varphi\in w$, by Lemma \ref{lemma: fatti basilari logica modale}, giving raise to a contradiction with the fact that $w$ is consistent. Observe that we have also proved that $s^{-}\not\vdash\neg\varphi$, hence by Lindenbaum Lemma there exists a maximally consistent set $s$ such that $s^{-}\subseteq s$ and $\neg\varphi\not\in s$. By maximality, we have that either $\varphi\in s$ or $\Juno\varphi\in s$. To show that $w\mathcal{R} s$, suppose $\Box\gamma\in w$, for some $\gamma\in Fm$, then by Lemma \ref{lemma: proprieta mondi massimali consistenti - 2} $\Jdue\Box\gamma\in w$, hence, by (M3), $\Box\Jdue\gamma\in w$, and by construction $\Jdue\gamma\in s^{-}\subseteq s$, thus $\gamma\in s$ (by Lemma \ref{lemma: proprieta mondi massimali consistenti - 2}), showing that $w\mathcal{R} s$. Finally, let us show that $\varphi$ is meaningful in $w$. Since $\varphi\in w$, then $\Jdue\varphi\in w$ and $+\varphi\in w$ (as $\vdash_{\B}\Jdue\varphi\to +\varphi$), namely that $\diamond\varphi$ is meaningful in $w$ and so is $\varphi$ (Lemma \ref{lemma: sulle formule sensate nei mondi massimali}).
\end{proof}

We are ready to define the concept of \emph{canonical model}.

\begin{definition}\label{def: modello canonico}
	The \emph{canonical model} for $\B^\square$ is a model $\mathcal{M}=\pair{\mathcal{W},\mathcal{R}, v}$ where $\mathcal{W}$ is the set of all maximally consistent sets of formulas, $\mathcal{R}$ is the canonical relation for $\B^\square$ and $v$ is defined as follows: 
	\begin{itemize}
		\item $v(w,x)=1$ if and only if $x\in w$;
		\vspace{2pt}
		\item $v(w,x)=0$ if and only if $\neg x\in X$;
		\vspace{2pt}
		\item $v(w,x)=\ant$ if and only if $\Juno x\in w$, 
	\end{itemize}
	for every $w\in\mathcal{W}$ and propositional variable $x$.
\end{definition}

\begin{lemma}[Truth Lemma]\label{lemma: truth lemma}
	Let $\mathcal{M}=\pair{\mathcal{W},\mathcal{R}, v}$ be the canonical model. Then, for every formula $\varphi\in Fm$ and every $w\in\mathcal{W}$, the following hold: 
	
	\begin{enumerate}
		\item $v(w,\varphi)=1$ if and only if $\varphi\in w$;
		\vspace{2pt}
		\item $v(w,\varphi)=0$ if and only if $\neg \varphi\in w$;
		\vspace{2pt}
		\item $v(w,\varphi)=\ant$ if and only if $\Juno\varphi\in w$.
	\end{enumerate}
\end{lemma}

\begin{proof}
	By induction on the length of the formula $\varphi$. We just show (1) for the inductive step when $\varphi=\Box\psi$, for some $\psi\in Fm$.\\
	\noindent
	Observe that $v(w,\Box\psi) = 1$ iff $\psi$ is meaningful in $w$ (i.e. $v(w,\psi)\neq\ant$) and $\forall s$ s.t. $w\mathcal{R}s$, $v(s,\psi) = 1$, thus, by induction hypothesis, iff $\Juno\psi\not\in w$ and $\psi\in s$ $\forall s$ s.t. $w\mathcal{R}s$. \\
	\noindent
	$(\Rightarrow) $ Suppose, by contradiction, that $v(w,\Box\psi) = 1$ but $\Box\psi\not\in w$, hence, by maximality of $w$, $\Juno\Box\psi\in w$ or $\neg\Box\psi\in w$. Since $\psi$ is meaningful in $w$, so is $\Box\psi$ (Lemma \ref{lemma: sulle formule sensate nei mondi massimali}), thus $\Juno\Box\psi\not\in w$. So, $\neg\Box\psi\in w$, i.e. $\diamond\neg\psi\in w$, hence, by Existence Lemma \ref{lemma: di esistenza}, there exists $s'\in\mathcal{W}$ such that $w\mathcal{R}s'$ such that $\neg\psi\in s'$ or $\Juno\psi\in s'$, which implies, by induction hypothesis, that $v(s',\psi) = 0$ or $v(s',\psi)=\ant$, a contradiction.\\
	\noindent
	$(\Leftarrow) $ Let $\Box\psi\in w$. Then $\Jdue\Box\psi\in w$ (by Lemma \ref{lemma: proprieta mondi massimali consistenti - 2}) and $+\Box\psi\in w$ (since $\vdash_{\B}\Jdue\gamma\to +\gamma$). This means that $\Box\psi$ is meaningful in $w$, hence so is $\psi$. Moreover, for every $s\in\mathcal{W}$ such that $w\mathcal{R}s$, we have that $\psi\in s$ (by definition of $\mathcal{R}$), hence, by induction hypothesis, $v(s,\psi) = 1$, from which $v(w,\psi) = 1$. 
\end{proof}

\begin{theorem}[Completeness]\label{th: completezza Bochvar modale}
	$\Gamma\vdash_{\B}\varphi$ if and only if $ \Gamma\models^l_{\B}\varphi$.
\end{theorem}
\begin{proof}
	$(\Rightarrow)$ It is easily checked that all the axioms are sound and the rules preserve soundness.\\
	\noindent
	$(\Leftarrow)$ Suppose $\Gamma\not\vdash\varphi$. Then $\Gamma$ is a consistent set of formulas, therefore, by the Lindenbaum Lemma \ref{lemma: Lindenbaum}, there exist a maximally consistent set $s$ such that $\Gamma\subseteq s$ and $\varphi\not\in s$, hence, by Truth Lemma \ref{lemma: truth lemma}, there exists a canonical countermodel, namely $(\mathcal{W},\mathcal{R},v)$, with $v(s,\gamma) = 1$ for all $\gamma\in\Gamma$ and $v(s,\varphi) \neq 1$.
\end{proof}

In order to prove decidability for $\B^\square$, we employ the filtration technique (see \cite[pp. 77-80]{BlackburnBook}). First we need to provide an extended notion of closure under subformulas. 

\begin{definition}\label{def.: chiusura per sottoformule}
	A set of formulas $\Sigma$ is closed under subformulas if $\forall\varphi,\psi\in \Sigma$:
	\begin{enumerate}
		\item if $\varphi\circ\psi\in \Sigma$ for any binary connective $\circ$, then $\varphi,\psi\in \Sigma$;
		\item if $\neg\varphi\in \Sigma$ or $\Jdue\varphi\in \Sigma$, then $\varphi\in \Sigma$;
		\item if $\square\varphi\in \Sigma$, then $\varphi\in \Sigma$ and $+\varphi\in \Sigma$; 
	\end{enumerate}

\end{definition}

Notice that if a set of formulas is finite, its closure under subformulas is still finite.

\begin{definition}
	Let $\pair{W,R, v}$ be a model and $\Sigma$ be a finite set of formulas closed under subformulas. This set induces an equivalence relation over $W$ defined as follows: $w \equiv_{\Sigma} s$ iff $\forall \varphi\in\Sigma (v(w,\varphi)=1$ iff $v(s,\varphi)=1)$.    
\end{definition}

When the reference set $\Sigma$ is clear from the context, we denote the equivalence class $[w]_{\equiv_\Sigma}$ simply by $[w]$.

\begin{definition}\label{def.: filtrazione}
	Let $M = \pair{W,R, v}$ be a model and $\Sigma$ be a finite set of formulas closed under subformulas. The filtration of $M$ through $\Sigma$ is the model $\pair{W^f,R^f, v^f}$ defined as:
	\begin{enumerate}
		\item $W^f = W/$$\equiv_{\Sigma}$; 
		\item $[w]R^f[s]$ iff $\exists w'\in [w],s'\in [s]$ s.t. $w'Rs'$;
		\item $v^f([w],p) = 1$ iff $v(w,p)=1$, for all variables $p \in \Sigma$. 
	\end{enumerate}
\end{definition}

\begin{lemma}\label{lemma: filtrations}
	Let $\pair{W^f,R^f, v^f}$ be a filtration of $M = \pair{W,R, v}$ through $\Sigma$. For all $\varphi \in \Sigma, w \in W$, it holds $v(w,\varphi)=1$ iff $v^f([w],\varphi)=1$.
\end{lemma}

\begin{proof}
	By induction on the complexity of $\varphi \in \Sigma$. The Boolean cases are straightforward. Let $\varphi = \Jdue\psi$, for some $\psi\in Fm$. $v(w,\Jdue\psi)=1$ iff $v(w,\psi)=1$ iff, by induction hypothesis, $v^f([w],\psi)=1$ iff $v^f([w],\Jdue\psi)=1$. Notice that by closure $\psi\in \Sigma$.
	
	Let $\varphi = \square\psi$, for some $\psi\in Fm$. Suppose $v(w,\square\psi)=1$, which means that $v(w,+\psi)=1$ and for all $s\in W$ s.t. $wRs, v(s,\psi)=1$. Now $+\psi := \Jdue\psi \lor \Jdue\neg\psi$, by the Boolean cases and the previous one we conclude $v^f([w],+\psi)=1$. By definition of filtration, $[w]R^f[s]$, and by induction hypothesis $v^f([s],\psi)=1$. Since this covers all the successors of $[w]$, then $v^f([w],\square\psi)=1$. Notice that by closure of $\psi$ under subformula, $\Jdue\psi \lor \Jdue\neg\psi\in \Sigma$. The other direction follows similarly.
\end{proof}

\begin{theorem}\label{th.: filtrazioni soddisfacilitÃ }
	If a formula $\varphi$ is satisfiable in a model, it is satisfiable in a finite model.
\end{theorem}

\begin{proof}
	Let $\varphi$ be satisfied by a model $M = \pair{W,R, v}$, and let $\Sigma$ be the closure under subformulas of $\{\varphi\}$. $\Sigma$ is finite. Now consider the filtration $M^f_\Sigma = \pair{W^f,R^f, v^f}$ of $M$ through $\Sigma$. By theorem \ref{lemma: filtrations}, $M^f_\Sigma$ satisfies $\varphi$. Consider the mapping $g: W^f \to \mathscr{P}(\Sigma)$ s.t. $g([w])=\{\psi | v(w,\psi)=1\}$. By definition of $\equiv_\Sigma$, $g$ is well-defined and injective. Denoting by $card(X)$ the cardinality of a set $X$, we have $card(W^f)\leq card(\mathscr{P}(\Sigma))=2^{card(\Sigma)}$.
\end{proof}

\begin{corollary}[Decidability]\label{th: Decidibilita}
	The logic $\B^{\Box}$ is decidable.
\end{corollary}

\section{Modal $\PWKe$ logic}\label{sec: modal PWK}

The modal extension of the propositional logic $\PWKe$ is defined over the same formula algebra $\Fm$ of $\B^{\Box}$. The substantial (semantical) difference between $\PWKe^{\Box} $ and $\B^{\Box}$ concern the interpretation of the modal formulas $\Box\varphi$, which follows the choice of the different truth-set in $\PWKe$: namely a modal formula $\Box\varphi$ will hold in a state $w$ iff it will also hold in all the related states $s$, namely in those the formula is not false.

The logic $\pair{\Fm,\vdash_{\MPWK}}$ is the consequence relation induced by the following Hilbert-style axiomatization.

\vspace{5pt}
\noindent
\textbf{Axioms}

\begin{itemize}
	\item the axioms for $\PWKe$ introduced in Definition \ref{def: assiomatizzazione PWK}; 
	\vspace{3pt}
	\item[(P1)]  $\square(J_2\varphi\to J_2\psi)\to (\square J_2\varphi\to\square J_2\psi)$;
	\vspace{3pt}
	\item[(P2)]  $\Box\varphi\leftrightarrow \Box\neg\Jzero \varphi$;
	\vspace{3pt}
	\item[(P3)]  $+\varphi\leftrightarrow +\square \varphi$.

\end{itemize}

\vspace{4pt}
\noindent
\textbf{Deductive rules}

\begin{itemize}
 \item ($\rho$-B13) $\Jdue\varphi\leftrightarrow\Jdue\psi,\Jzero\varphi\leftrightarrow\Jzero\psi\vdash\varphi\equiv\psi$;
 \vspace{5pt}
        \item (PWKAlg3) $\varphi\dashv\vdash\neg\Jzero\varphi\leftrightarrow 1$.
	\vspace{5pt}
	
	\item (N): if $\vdash \varphi $ then $\vdash \square\varphi$. 
	
	

\end{itemize}
\vspace{5pt}

In this Section, by $\vdash$ we will mean $\vdash_{\MPWK}$.

\begin{lemma}\label{lemma: K per PWK su esterne}
    For $\alpha,\beta$ external formulas, the following is a theorem of \MPWK:
    \begin{itemize}
        \item[\emph{(\text{BK})}] $\square(\alpha\to\beta)\to (\square\alpha\to\square\beta)$
    \end{itemize}    
\end{lemma}
\begin{proof}
It is the same of Lemma \ref{lemma: K per Bochvar su esterne}.
\end{proof}

The Deduction theorem holding for $\PWKe$ can be actually extented to its modal version.

\begin{theorem}[Deduction Theorem]\label{lemma: deduzione PWK modale}
For the logic \MPWK, it holds that $\Gamma\vdash\varphi$ iff there exist some formulas $\gamma_1,\dots,\gamma_n\in \Gamma$ such that $\vdash\neg\Jzero\gamma_1\wedge\dots\wedge\neg\Jzero\gamma_n\to \neg\Jzero\varphi$.
\end{theorem}
\begin{proof}
($\Rightarrow$). By induction on the length of the derivation of $\varphi$ from $\Gamma$. \\
\noindent
\textit{Basis.} If $\varphi$ is an axiom ($\vdash\varphi$), then by Lemma \ref{lemma: fatti basilare PWKe} we have $\vdash\neg\Jzero\varphi$.\\
\noindent
\textit{Inductive step.} We just show the case of the rule (N). Let $\varphi=\Box\psi$, for some $\psi\in Fm$, and $\Gamma\vdash\Box\psi$, and the last deduction rule applied is (N), hence it holds $\vdash\psi$. Therefore we have $\vdash\Box\psi$, hence $\vdash\neg\Jzero\Box\psi$, by Lemma \ref{lemma: fatti basilare PWKe}. By induction hypothesis, there exist some formulas $\gamma_1,\dots,\gamma_n\in \Gamma$ such that $\vdash\neg\Jzero\gamma_1\wedge\dots\wedge\neg\Jzero\gamma_n\to \neg\Jzero\psi$. Observe that (by classical logic, Lemma \ref{lemma: fatti basilare PWKe}) $\vdash\neg\Jzero\Box\psi\to(\neg\Jzero\gamma_1\wedge\dots\wedge\neg\Jzero\gamma_n\to \neg\Jzero\Box\psi)$, hence, by modus ponens (on external formulas, see Lemma \ref{lemma: fatti basilare PWKe}), $\vdash\neg\Jzero\gamma_1\wedge\dots\wedge\neg\Jzero\gamma_n\to \neg\Jzero\Box\psi $.\\
\noindent
($\Leftarrow$). Suppose $\vdash\neg\Jzero\gamma_1\wedge\dots\wedge\neg\Jzero\gamma_n\to \neg\Jzero\varphi$. By Lemma \ref{lemma: fatti basilare PWKe}, we have $\Gamma\vdash\neg\Jzero \gamma$, $\forall\gamma\in\Gamma$, therefore $\Gamma\vdash\neg\Jzero\gamma_1\wedge\dots\wedge\neg\Jzero\gamma_n$, thus applying modus ponens (on external formulas by Lemma \ref{lemma: fatti basilare PWKe}), we get $\Gamma\vdash\neg\Jzero\varphi$. Again, by Lemma \ref{lemma: fatti basilare PWKe} we have $\Gamma\vdash\varphi$ (as $\neg\Jzero\varphi\vdash\varphi$).
\end{proof}

\begin{lemma}\label{lemma: fatti basilari logica modale PWK}
	In the logic \MPWK it holds: 
	\begin{enumerate}
		\item[] $\square(J_i\varphi\land J_k\psi) \leftrightarrow \square J_i\varphi\land\square J_k\psi$ for $i,k \in \{ 0,1,2 \}$.
	\end{enumerate}
\end{lemma}
\begin{proof}
    The proof is identical to that of Lemma \ref{lemma: fatti basilari logica modale}. Observe that even though \MPWK does not possess a full modus ponens, in this proof we are dealing exclusively with external formulas for which full modus ponens holds (see Lemma \ref{lemma: fatti basilare PWKe}).
    
\end{proof}

\subsection{Semantics}
The semantics for \MPWK employs the weak three-valued Kripke models of Definition \ref{def: Kripke model}. In particular we consider the following subclass:

\begin{definition}\label{def: semantica Box per PWK}
	A \emph{PWK-Kripke} model is a weak three-valued Kripke model $\pair{W,R, v}$ such that $v$ evaluates formulas of the form $\Box\varphi$ according to:
	\begin{itemize}
		\item[(1)] $v(w,\Box\varphi)= 1$ iff $v(w,\varphi)\neq\ant$ and $v(s,\varphi)\neq0$ for every $s\in W$ such that $wRs$. 
		\item[(0)] $v(w,\Box\varphi)= 0$ iff $v(w,\varphi)\neq\ant$ and there exists $s\in W$ such that $wRs$ and $v(s,\varphi)=0$. 
		\item[($\ant$)] $v(w,\Box\varphi)= \ant$ iff $v(w,\varphi)=\ant$. 
	\end{itemize}
\end{definition}

Within this Section by model we intend PWK-Kripke model.

\begin{remark}\label{rem: semantica Diamond per PWK}
Observe that, in every model $\pair{W,R, v}$ with $w\in W$, it holds that $v(w,\diamond\varphi) = 1$ iff $v(w,\varphi)\neq\ant$ and there exists $s\in W$ such that $wRs$ and $v(s,\varphi)=1$, $v(w,\diamond\varphi) = 0$ iff $v(w,\varphi)\neq\ant$ and $v(s,\varphi)\neq1$ for every $s\in W$ such that $wRs$. 
\end{remark}

We denote by $\models^{l}_{\MPWK}$ the \emph{local} modal PWK external logic on the class of all frames obtained by taking $\{ 1,\ant \}$ as designated values, that is:

\begin{definition}\label{def: conseguenza logica pwk}
        $\Gamma \vDash^{l}_{\MPWK} \varphi$ iff for all models $\pair{W,R, v}$ and all $w\in W$, if $v(w,\gamma)\neq0, \forall\gamma\in\Gamma$, then $v(w,\varphi)\neq0$.
\end{definition}

In the following we omit both the subscript $\MPWK$ and the superscript $l$, simply writing $\vDash$ (hopefully with no danger of confusion, since we will not deal with \emph{global} modal logics in the present paper). Notice that the notion of satisfiability in $\MPWK$ differs from $\B^{\square}$, according to the difference at the propositional level between $\PWKe$ and $\B$.

\begin{definition}\label{def: soddisfacibilitÃ  e validitÃ  PWK}
        A formula $\varphi$ is satisfied (valid) in a model $\pair{W,R, v}$ if $v(w,\varphi)\neq0$, for some (all) $w\in W$.
\end{definition}

Validity in a frame and in a class of frames follow Definition \ref{def: soddisfacibilitÃ  e validitÃ   Bochvar}, with the exception that we consider only PWK-Kripke models built on a frame. 

\subsection{Completeness and Decidability}

The notion of consistency has to be changed for $\MPWK$, because the sublogic $\PWK$ is paraconsistent. Therefore we substiture consistent (and maximally consistent) sets with non-trivial ones.

\begin{remark}\label{rem: consistenza in PWK modale}
A set $\Gamma\subset Fm$ is \emph{non-trivial} if there is no formula $\varphi\in Fm$, such that $\Gamma\vdash J_i\varphi$ and $\Gamma\vdash\neg J_i\varphi$, for any $i\in \{0,1,2\}$. It is called trivial otherwise.
\end{remark}

\begin{definition}\label{rem: consistenza massimale in PWK modale}
	A non-trivial set $\Gamma$ is \emph{maximally non-trivial} whenever $\Gamma\subset \Gamma'$ implies that $\Gamma'$ is trivial.
\end{definition}

A \emph{meaningful} formula is the same of Definition \ref{def: formula sensata in un mondo}) for modal Bochvar logic, by just considering that $\vdash$ here refers to $\MPWK$ (and not to $\B^{\Box}$).

The following is the analogous of Lemma \ref{lemma: proprieta mondi massimali consistenti - 1}, for $\MPWK$ (indeed, the first claims coincide).

\begin{lemma}\label{lemma: proprieta mondi massimali consistenti PWK}
	Let $w$ be a maximally non-trivial set of formulas, then:
	\begin{enumerate}
		\item if $\neg\varphi\not\in w$ and all the variables occurring in $\varphi$ are meaningful in $w$, then $\varphi\in w$; 
		\item if all the variables of $\varphi$ are covered, then $\varphi$ is meaningful in $w$;
		\item if $\vdash\varphi$ then $\neg\varphi\not\in w$;
		\item if $\vdash\varphi$ and every variable in $\varphi$ is meaningful then $\varphi\in w$.
	\end{enumerate}
\end{lemma}
\begin{proof}
(3). Suppose that $\vdash\varphi$ and, by contradiction, that $\neg\varphi\in w$. By Lemma \ref{lemma: fatti basilare PWKe} $\vdash\neg\Jzero\varphi$, thus $w\vdash\Jzero\varphi$. On the other hand $w\vdash\Jzero\varphi$ (by Lemma \ref{lemma: fatti basilare PWKe}). But this implies that $w$ is trivial (by Remark \ref{rem: consistenza in PWK modale}).\\
\noindent
(4) follows from the previous. 
\end{proof}

\begin{lemma}\label{lemma: se inconsistente allora ...}
Let $\Gamma$ be a non-trivial set of formulas. If $\Gamma\cup\{\varphi\}$ is trivial then $\Gamma\vdash\neg\varphi$.
\end{lemma}
\begin{proof}
Let $\Gamma$ be a non-trivial set of formulas. Suppose  $\Gamma\cup\{\varphi\}$ is trivial. Therefore $\Gamma\cup\{\varphi\}\vdash 0$; by Theorem \ref{lemma: deduzione PWK modale}, there exist formulas $\gamma_{1},\dots,\gamma_{n}\in \Gamma$ s.t. $\vdash\neg\Jzero\gamma_1\wedge\dots\wedge\neg\Jzero\gamma_n\land\neg\Jzero\varphi\to\neg\Jzero 0$, where the non-triviality of $\Gamma$ assures that $\Jdue\varphi$ actually appears in the antecedent. We have $\vdash_{\PWKe}\neg\Jzero 0\leftrightarrow 0$ by Lemma \ref{lemma: fatti basilare PWKe}, hence $\vdash\neg\Jzero\gamma_1\wedge\dots\wedge\neg\Jzero\gamma_n\land	\neg\Jzero\varphi\to 0$, therefore $\vdash\neg\Jzero\gamma_1\wedge\dots\wedge\neg\Jzero\gamma_n\to\Jzero\varphi$. Since $\vdash\Jzero\varphi\to\neg\Jzero\neg\varphi$, by classical logic $\vdash\neg\Jzero\gamma_1\wedge\dots\wedge\neg\Jzero\gamma_n\to\neg\Jzero\neg\varphi$. Thus, by Theorem \ref{lemma: deduzione PWK modale}, $\gamma_1\wedge\dots\wedge\gamma_n\vdash\neg\varphi$, hence by monotonicity $\Gamma\vdash\neg\varphi$.
\end{proof}

Lindenbaum's lemma for $\MPWK$ has a slightly different form from that of $\B^{\Box}$.

\begin{lemma}[Lindenbaum's Lemma for \MPWK]\label{lemma: Lindenbaum per PWK}
	Let $\Gamma$ be a non-trivial set of formulas such that $\Gamma\not\vdash\varphi$, for some $\varphi\in Fm$, then there exists a maximally non-trivial set of formulas $w$ such that $\Gamma\subseteq w$ and $\neg\varphi\in w$.
\end{lemma}
\begin{proof}
	Suppose $\Gamma\not\vdash\varphi$. Let $\psi_1,\psi_2,\psi_3,\dots$ be an enumeration of the formulas of $\MPWK$. We define the sets inductively:
 $$\Gamma_{0} = \Gamma \cup \{ \neg\varphi \}.$$
 \begin{equation*}
  \Gamma_{i+1} = \left \{
  \begin{aligned}
    &\Gamma_i \cup \{ \psi_i \}, && \text{if non-trivial} \\
    &\Gamma_i \cup \{ \neg\psi_i \}, && \text{if non-trivial} \\
    &\Gamma_i \cup \{ \Juno\psi_i \}, && \text{otherwise}
  \end{aligned} \right.
\end{equation*} 
$$w = \bigcup\limits_{i \in \mathbb{N}} \Gamma_i.$$

By construction $w$ is maximal, $\Gamma \subseteq w$ and $\neg\varphi\in w$. We prove the non-triviality of $w$ by induction on $n\in\mathbb{N}$. For the base step, since $\Gamma\not\vdash\varphi$, by Lemma \ref{lemma: se inconsistente allora ...} $\Gamma_{0} = \Gamma \cup \{ \neg\varphi \}$ is non-trivial. For the inductive step, suppose $\Gamma_i$ is non-trivial, while both $\Gamma_i \cup \{ \psi_i \}$ and $\Gamma_i \cup \{ \neg\psi_i \}$ are trivial. Therefore $\Gamma_{i+1} = \Gamma_i \cup \{ \Juno\psi_i \}$. By the same lemma the previous facts yield $\Gamma_i\vdash\neg\psi_i$ and $\Gamma_i\vdash\psi_i$. By Lemma \ref{lemma: fatti basilare PWKe} we obtain $\Gamma_i\vdash\Jdue\neg\psi_i\lor\Juno\neg\psi_i$ and $\Gamma_i\vdash\Jdue\psi_i\lor\Juno\psi_i$, of which the former can be rewritten by the same lemma as $\Gamma_i\vdash\Jzero\psi_i\lor\Juno\psi_i$. Since $\Jzero\psi_i,\Juno\psi_i,\Jdue\psi_i$ are pairwise contradictory, by classical logic we conclude $\Gamma_i\vdash\Juno\psi_i$. Since $\Gamma_i$ is taken as non-trivial, this implies $\Gamma_i\nvdash\neg\Juno\psi_i$. By Lemma \ref{lemma: se inconsistente allora ...}, we conclude that $\Gamma_i \cup \{ \Juno\psi_i \} = \Gamma_{i+1}$ is non-trivial.
\end{proof}

\begin{definition}\label{def: relazione canonica per MPWK}
	Let $\mathcal{W}$ be the set of all maximally non-trivial set of formulas. Then the \emph{canonical relation} $\mathcal{R}\subset \mathcal{W}\times \mathcal{W}$ for \MPWK is defined, for every $w,s\in \mathcal{W}$ as: 
	$$w\mathcal{R} s \text{ iff }\forall\varphi\in Fm \text{ s.t. } \Box\varphi\in w \text{ then } \neg\varphi\notin s. $$	
\end{definition}

\begin{lemma}[Existence Lemma]\label{lemma: di esistenza per MPWK}
	For every maximally non-trivial set of formulas $w$ such that $\diamond\varphi\in w$ (for some $\varphi\in Fm$) then $\varphi$ is meaningful in $w$ and there exists a maximally non-trivial set $s\in\mathcal{W}$ such that $w\mathcal{R}s$ and $\varphi\in s$. 
\end{lemma}
\begin{proof}
    Let $\diamond\varphi\in w$ and consider the set
$$ s^{-} = \{\psi \;|\; \Box\neg\Jzero\psi\in w\}. $$
Observe that $s^-\neq\emptyset$, in fact by Lemma \ref{lemma: fatti basilare PWKe}, $\vdash\neg\Jzero 1$, therefore by (N) $\vdash\Box\neg\Jzero 1$. Since $w$ is maximally non-trivial, $\Box\neg\Jzero 1\in w$, hence $1\in s^-$. Suppose by contradiction that $s^-$ is trivial. Therefore $s^-$ derives every formula, in particular $s^-\vdash\neg\varphi$. By Theorem \ref{lemma: deduzione PWK modale}, there are $\psi_1,\dots,\psi_n\in s^-$ s.t. $\vdash\neg\Jzero\psi_1\land\dots\land\neg\Jzero\psi_n\to\neg\Jzero\neg\varphi$. By (N) $\vdash\Box(\neg\Jzero\psi_1\land\dots\land\neg\Jzero\psi_n\to\neg\Jzero\neg\varphi)$, and using (BK) and modus ponens (since the formulas considered here are external) we get $\vdash\Box(\neg\Jzero\psi_1\land\dots\land\neg\Jzero\psi_n)\to\Box\neg\Jzero\neg\varphi$. Now Lemma \ref{lemma: fatti basilari logica modale PWK} can be generalized to $\vdash\Box\neg\Jzero\psi_1\land\dots\land\Box\neg\Jzero\psi_n\to\Box(\neg\Jzero\psi_1\land\dots\land\neg\Jzero\psi_n)$, obtaining, by transitivity, $\vdash\Box\neg\Jzero\psi_1\land\dots\land\Box\neg\Jzero\psi_n\to\Box\neg\Jzero\neg\varphi$. Observe that $\Box\neg\Jzero\psi_1,\dots,\Box\neg\Jzero\psi_n\in w$, therefore $\Box\neg\Jzero\neg\varphi\in w$ (as $w$ is maximally non-trivial). By (P2) $\Box\neg\varphi\in w$, which can be rewritten as $\neg\diamond\varphi\in w$, contradicting the non-triviality of $w$. 

Notice that we have also proved that in particular $s^-\nvdash\neg\varphi$. We can now apply Lindembaum's lemma and extend $s^-$ to a maximally non-trivial $s\supseteq s^-$ s.t. $\varphi\in s$, hence $\varphi$ is meaningful in $s$. Finally we show that $w\mathcal{R}s$ according to the canonical relation: let for arbitrary $\chi\in Fm,\Box\chi\in w$, then by (P2) $\Box\neg\Jzero\chi\in w$, so $\chi\in s^-\subseteq s$ and by non-triviality $\neg\chi\notin s$.
\end{proof}

We now adapt the definition of canonical model to \MPWK:
\begin{definition}
\label{def: modello canonico MPWK}
	The \emph{canonical model} for \MPWK is a weak three-valued Kripke model \\
	\noindent
	$\mathcal{M}=\pair{\mathcal{W},\mathcal{R}, v}$ where $\mathcal{W}$ is the set of all maximally non-trivial sets of formulas, $\mathcal{R}$ is the canonical relation for \MPWK and $v$ is defined as follows: 
	\begin{itemize}
		\item $v(w,x)=1$ if and only if $x\in w$;
		\vspace{2pt}
		\item $v(w,x)=0$ if and only if $\neg x\in X$;
		\vspace{2pt}
		\item $v(w,x)=\ant$ if and only if $\Juno x\in w$, 
	\end{itemize}
	for every $w\in\mathcal{W}$ and propositional variable $x$.
\end{definition}

\begin{lemma}[Truth Lemma]\label{lemma: truth lemma per MPWK}
	Let $\mathcal{M}=\pair{\mathcal{W},\mathcal{R}, v}$ be the canonical model. Then, for every formula $\varphi\in Fm$ and every $w\in\mathcal{W}$, the following hold: 
	
	\begin{enumerate}
		\item $v(w,\varphi)=1$ if and only if $\varphi\in w$;
		\vspace{2pt}
		\item $v(w,\varphi)=0$ if and only if $\neg \varphi\in w$;
		\vspace{2pt}
		\item $v(w,\varphi)=\ant$ if and only if $\Juno\varphi\in w$.
	\end{enumerate}
\end{lemma}

\begin{proof}
	By induction on the length of the formula $\varphi$. We just show (1) for the inductive step when $\varphi=\Box\psi$, for some $\psi\in Fm$.\\
	\noindent
	Observe that $v(w,\Box\psi) = 1$ iff $\psi$ is meaningful in $w$ (i.e. $v(w,\psi)\neq\ant$) and $\forall s$ s.t. $w\mathcal{R}s$, $v(s,\psi) \neq 0$, thus, by induction hypothesis, iff $\Juno\psi\not\in w$ and $\neg\psi\notin s$ $\forall s$ s.t. $w\mathcal{R}s$. \\
	\noindent
	$(\Rightarrow) $ Suppose, by contradiction, that $v(w,\Box\psi) = 1$ but $\Box\psi\not\in w$, hence, by maximality of $w$, $\Juno\Box\psi\in w$ or $\neg\Box\psi\in w$. Since $\psi$ is meaningful in $w$, so is $\Box\psi$ (Lemma \ref{lemma: proprieta mondi massimali consistenti PWK}), thus $\Juno\Box\psi\not\in w$. So, $\neg\Box\psi\in w$, i.e. $\diamond\neg\psi\in w$, hence, by Existence Lemma \ref{lemma: di esistenza per MPWK}, there exists $s'\in\mathcal{W}$ such that $w\mathcal{R}s'$ such that $\neg\psi\in s'$, which implies, by induction hypothesis, that $v(s',\psi) = 0$, a contradiction.\\
	\noindent
	$(\Leftarrow)$ The proof is similar to one of the Truth Lemma \ref{lemma: truth lemma} for $\B^\square$. Observe that $\vdash_{\PWKe} \varphi\to\Jdue\varphi$ and $\vdash_{\PWKe} \Jdue\varphi\to+\varphi$, moreover maximally non-trivial sets are closed under unrestricted modus ponens.
\end{proof}

\begin{theorem}[Completeness]\label{th: completezza MPWK}
	$\Gamma\vdash_{\MPWK}\varphi$ if and only if $ \Gamma\models^{l}_{\MPWK}\varphi$.
\end{theorem}
\begin{proof}
	$(\Rightarrow)$ It is easily checked that all the axioms are sound and the rules preserve soundness.\\
	\noindent
	$(\Leftarrow)$ Suppose $\Gamma\not\vdash\varphi$. Then $\Gamma$ is a non-trivial set of formulas, therefore, by the Lindenbaum Lemma \ref{lemma: Lindenbaum per PWK}, there exist a maximally non-trivial set $s$ such that $\Gamma\subseteq s$ and $\neg\varphi\in s$, hence, by Truth Lemma \ref{lemma: truth lemma}, there exists a canonical countermodel, namely $(\mathcal{W},\mathcal{R},v)$, with $v(s,\gamma) = 1$ for all $\gamma\in\Gamma$ and $v(s,\varphi) = 0$.
\end{proof}

The decidability of $\MPWK$ is proved similarly to the case for $\B^\square$. Once again we make use of filtrations. The notion of closure under subformulas is the same given in Definition \ref{def.: chiusura per sottoformule}. The only essential change concerns the equivalence relation that induces the partition.

\begin{definition}
	Let $\pair{W,R, v}$ be a model and $\Sigma$ be a finite set of formulas closed under subformulas. This set induces an equivalence relation over $W$ such that: $w \equiv_\Sigma s$ iff $\forall \varphi\in\Sigma (v(w,\varphi)=1$ iff $v(s,\varphi)\neq0)$.    
\end{definition}

Again, when there is no risk of confusion we omit the reference set $\Sigma$ and denote the equivalence class $[w]_{\equiv_\Sigma}$ as $[w]$. We can adopt Definition \ref{def.: filtrazione} for filtration.

\begin{theorem}\label{th: filtrations for pwk}
	Let $\pair{W^f,R^f, v^f}$ be a filtration of $M = \pair{W,R, v}$ through $\Sigma$. For all $\varphi \in \Sigma, w \in W$, it holds $v(w,\varphi)\neq0$ iff $v^f([w],\varphi)\neq0$.
\end{theorem}

\begin{proof}
	By induction on the complexity of $\varphi \in \Sigma$. The Boolean cases and the case for $\varphi = \Jdue\psi$ follows Lemma \ref{lemma: filtrations}.
	
	Let $\varphi = \square\psi$, for some $\psi\in Fm$. When $v(w,\square\psi)=1$ the proof is similar to the case covered in Lemma \ref{lemma: filtrations}. Suppose $v(w,\square\psi)=\ant$: this holds iff $v(w,\psi)=\ant$ iff (by induction hypothesis) $v([w],\psi)=\ant$ iff $v([w],\square\psi)=\ant$.
\end{proof}

In accordance with the notion of satisfiability in \MPWK we reobtain the main theorem:

\begin{theorem}
	If a formula $\varphi$ is satisfiable in a model, it is satisfiable in a finite model.
\end{theorem}

\begin{proof}
	Identical to Theorem \ref{th.: filtrazioni soddisfacilitÃ }.
\end{proof}

\begin{corollary}[Decidability]\label{th: Decidibilita per PWK}
	The logic \MPWK is decidable.
\end{corollary}

\section{Extensions of modal weak Kleene logics}\label{sec: estensioni}

The aim of this Section is to axiomatize some extensions of both the modal logics $\B^{\Box}$ and $\MPWK$. In particular, we focus on those extensions whose semantical counterpart is given by reflexive, transitive and/or Euclidean models (clearly, the properties refer all to the relational part of models). To this end, consider the following formulas: 
$$ \text{(T$_e$) }\;\; \Box\Jdue\varphi\to\Jdue\varphi, $$
$$ \text{(4$_e$) }\; \;\;\Box\Jdue\varphi\to\Box\Box\Jdue\varphi, $$
$$ \text{(5$_e$) } \;\;\diamond\Jdue\varphi\to\Box\diamond\Jdue\varphi, $$

\vspace{5pt}
\noindent
which substantially consist of ``external versions''\footnote{The problem with standard modal formulas is the same with Bochvar (non-external) propositional logic, which is well-known as a logic without theorems, since every formula can be evaluated into $\ant$. Similarly, in the modal context using the internal formulas (T), (4) or (5) would provide an unsound axiomatization.
} of the standard modal formulas (T), (4) and (5). The formulas introduced above allows to capture some frame properties within $\B^{\square}$, as shown by the following.

\begin{proposition}\label{prop: proprietÃ  dei frame Bochvar}
	Let $\mathcal{F}=(W,R)$ be a frame. Then
	\begin{enumerate}
		\item $\mathcal{F}\vDash_{\B^{\square}} $ \emph{T$_e$} iff $R$ is reflexive; 
		\item $\mathcal{F}\vDash_{\B^{\square}} $ \emph{4$_e$} iff $R$ is transitive;
            \item $\mathcal{F}\vDash_{\B^{\square}} $ \emph{5$_e$} iff $R$ is euclidean.
	\end{enumerate}
\end{proposition}
\begin{proof}
    (1) $(\Rightarrow)$ Suppose $\mathcal{F}\models \text{T$_e$}$ and, for arbitrary $w\in W$, let $X= \{s\in W \;|\; wRs\}$. Consider the valuation $v(s,p) = 1$ iff $s\in X$, for any propositional variable $p$. It follows that $v(s,\Jdue p) = 1$, and since $X$ is the set of successors of $w$, we have $v(w,\Box\Jdue p) = 1$. By assumption $\mathcal{F}\models \text{T$_e$}$, therefore in particular $v(w,\Box\Jdue p\to\Jdue p)= 1$. It follows that $v(w,\Jdue p) = 1$, thus $v(w,p) = 1$, so, by definition, $w\in X$, that is $wRw$, showing that $R$ is reflexive. \\
	\noindent
	$(\Leftarrow)$ It is immediate to check that T$_e$ is valid in every reflexive frame.\\
 
    (2) $(\Rightarrow)$ Assume $\mathcal{F}\models$ 4$_e$, let $w\in W$ such that $w R s'$ and $s'\mathcal{R}t$, for some $s',t\in W$. Consider the valuation $v(s,p) = 1$ iff $wRs$, for every $s\in W$ and any propositional variable $p$. This implies that $v(s,\Jdue p) = 1$ for every $wR s$, thus $v(w,\Box\Jdue p) = 1$ (observing that $v(w,\Jdue p)\neq\ant$. Since $\mathcal{F}\models $ 4$_e$ then $v(w,\Box\Box\Jdue p) = 1$. Since $w R s'$, then $v(s', \Box\Jdue p) = 1$, therefore $v(t,\Jdue p) = 1$ (since $s' R t$), thus $v(t, p) = 1$. This implies that $w R t$, i.e. $R$ is transitive as desired.

	\noindent
	$(\Leftarrow)$ It is immediate to check that 4$_e$ is valid in every transitive frame.\\
 
    (3) $(\Rightarrow)$ Suppose $\mathcal{F}$ to be non-euclidean, therefore for some $w,s',s''\in W, wRs', wRs''$ but $\pair{s',s''}\notin R$. Define the valuation $v$ such that for an arbitrary variable $p $, $ v(w,p)=v(s'',p)=1$, while $v(s',p)=0$ and for all $t$ s.t. $s'Rt,v(t,p)=0$. It follows that $v(w,\diamond\Jdue p)=1$ but $v(s',\diamond\Jdue p)=0$, therefore $v(w,\Box\diamond\Jdue p)=0$, hence $v(w,\diamond\Jdue p\to\Box\diamond\Jdue p)=0$. This countermodel proves $\mathcal{F}\nvDash \text{5}_e.$ \\
	\noindent
	$(\Leftarrow)$ It is immediate to check that 5$_e$ is valid in every euclidean frame.
\end{proof}

In the case of the logic $\MPWK$, the same frame properties are expressed by the standard modal formulas: 
$$ \text{(T) } \Box\varphi\to\varphi $$
$$ \text{(4) } \Box\varphi\to\Box\Box\varphi $$
$$ \text{(5) } \diamond\varphi\to\Box\diamond\varphi $$

\begin{proposition}\label{prop: proprietÃ  dei frame PWK}
	Let $\mathcal{F}=(W,R)$ be a frame. Then
	\begin{enumerate}
		\item $\mathcal{F}\vDash_{\MPWK}$ \emph{(T)} iff $R$ is reflexive; 
		\item $\mathcal{F}\vDash_{\MPWK}$ \emph{(4)} iff $R$ is transitive;
            \item $\mathcal{F}\vDash_{\MPWK} $\emph{(5)} iff $R$ is euclidean.
	\end{enumerate}
\end{proposition}
\begin{proof}
The proofs runs similarly as Proposition \ref{prop: proprietÃ  dei frame Bochvar}.
\end{proof}

The correspondences established by Propositions \ref{prop: proprietÃ  dei frame Bochvar} and \ref{prop: proprietÃ  dei frame PWK} allow us to immediately prove completeness for some extensions of $\B^{\square}$ and $\MPWK$. For an axiomatic calculus L, let LAx$_1\dots$Ax$_n$ be the logic obtained by adding (Ax$_1),\dots$,(Ax$_n$) to L. Moreover we use the following abbreviations: S4 := T + 4, S5 := T + 5, S4$_e$ := T$_e$ + 4$_e$, S5$_e$ := T$_e$ + 5$_e$.

\begin{theorem}
    The relation $\mathcal{R}$ of the canonical models\footnote{The canonical model for a certain extension $\B^{\square}$Ax of $\B^{\square}$ differs from Definition \ref{def: modello canonico} only for the set of worlds, which now consist not of all maximal consistent sets (w.r.t. $\B^{\square}$), but only of the maximal consistent theories of $\B^{\square}$Ax. The canonical model for $\MPWK$Ax is adapted from Definition \ref{def: modello canonico MPWK} in a similar fashion.} for the following logics have the properties:
    \begin{itemize}[align=left]
        \item In $\B^{\square}\emph{T}_e$ and $\MPWK\emph{T}$ $\mathcal{R}$ is reflexive;
        \item In $\B^{\square}4_e$ and $\MPWK4$ $\mathcal{R}$ is transitive;
        \item In $\B^{\square}5_e$ and $\MPWK5$ $\mathcal{R}$ is euclidean;
        \item In $\B^{\square}\emph{S}4_e$ and $\MPWK\emph{S}4$ $\mathcal{R}$ is reflexive and transitive;
        \item In $\B^{\square}\emph{S}5_e$ and $\MPWK\emph{S}5$ $\mathcal{R}$ is an equivalence relation.
    \end{itemize}
\end{theorem}
\begin{proof}
    It follows from Propositions \ref{prop: proprietÃ  dei frame Bochvar} and \ref{prop: proprietÃ  dei frame PWK}.
\end{proof}

That the accessibility relation of the canonical model has the desired properties is enough to obtain the following completeness results:

\begin{corollary}\label{cor: completezza estensioni Bochvar}
    The following hold:
    \begin{itemize}
        \item $\B^{\square}\emph{T}_e$ is complete with respect to the class of reflexive frames;
        \item $\B^{\square}4_e$ is complete with respect to the class of transitive frames;
        \item $\B^{\square}5_e$ is complete with respect to the class of euclidean frames;
        \item $\B^{\square}\emph{S}4_e$ is complete with respect to the class of reflexive and transitive frames;
        \item $\B^{\square}\emph{S}5_e$ is complete with respect to the class of frames whose relation is an equivalence.
    \end{itemize}
\end{corollary}
\begin{corollary}\label{cor: completezza estensioni PWK}
    The following hold:
    \begin{itemize}
        \item $\MPWK\emph{T}$ is complete with respect to the class of reflexive frames;
        \item $\MPWK4$ is complete with respect to the class of transitive frames;
        \item $\MPWK5$ is complete with respect to the class of euclidean frames;
        \item $\MPWK\emph{S}4$ is complete with respect to the class of reflexive and transitive frames;
        \item $\MPWK\emph{S}5$ is complete with respect to the class of frames whose relation is an equivalence.
    \end{itemize}
\end{corollary}

Notice that depending on the choice of the basic logic we obtain a different notion of completeness: in the case of $\B^{\square}$ the completeness is w.r.t. the logical consequence relation $\vDash_{\B^{\square}}$, in the case of \MPWK the relation is $\vDash_{\MPWK}$.

The decidability of $\B^{\square}$ and \MPWK established by Corollaries \ref{th: Decidibilita} and \ref{th: Decidibilita per PWK} immediately follows for their (finitely axiomatizable) axiomatic extensions.

\begin{theorem}
    For $\emph{E}\in\{ \emph{T}_e,4_e,5_e,\emph{S}4_e,\emph{S}5_e \}$, the logic $\B^{\square}\emph{E}$ is decidable. For $\emph{E}\in\{ \emph{T},4,5,\emph{S}4,\emph{S}5 \}$, the logic $\MPWK\emph{E}$ is decidable.
\end{theorem}
\begin{proof}
    We give the proof for $\B^{\square}$T$_e$, the others cases employs the same strategy. Using the same definitions of set closed under subformulas and filtration from Definitions \ref{def.: chiusura per sottoformule} and \ref{def.: filtrazione}, we have that Lemma \ref{lemma: filtrations} still holds. Suppose that $\varphi$ is satisfiable in a reflexive model $M$, therefore by Proposition \ref{prop: proprietÃ  dei frame Bochvar} T$_e$ is valid in $M$. Let $\Gamma=\{ \text{T}_e,\varphi \}$, $\Sigma$ its closure under subformulas, and consider the filtration $M^\Sigma$ of $M$ through $\Sigma$. By Lemma \ref{lemma: filtrations}, T$_e$ is valid in $M^\Sigma$, hence by Proposition \ref{prop: proprietÃ  dei frame Bochvar} $M^\Sigma$ is reflexive. Now we repeat Theorem \ref{th.: filtrazioni soddisfacilitÃ } using the filtration $M^\Sigma$ to prove that if $\varphi$ is satisfiable in reflexive model, it is satisfiable in a finite reflexive model of cardinality at most $2^{card(\Sigma)}$. We have the desired finite model property, which, together with the completeness theorem stated in Corollary \ref{cor: completezza estensioni Bochvar}, gives the decisability of $\B^{\square}$T$_e$.
\end{proof}

\section{Conclusions and future work}

In this work we have studied two modal expansions of the external weak Kleene logics $\B$ and $\PWKe$. The modalization yields two different operators $\Box$, which, without further inquiry, can be intended as generic necessary (in the alethic sense) operators. The distinction between the two is motivated by the difference in designated values within the propositional bases. Moreover, the logics $\B$ and $\PWKe$ have been first provided with new axiomatizations, which differ from their ones (by Finn-Grigolia and Segerberg, respectively) since they are obtained thanks to the recent results that state the algebraizability of $\B$ (\cite{Ignoranzasevera}) and $\PWKe$ (see Section \ref{sec: Appendice}).

The introduction of modalities to $\B$ and $\PWKe$ adds further expressive power to a language already capable of expressing a (sort of) truth predicate ($\Jdue$), allowing to speak about the ``truth'' of a formula. The interplay between $\Jdue$ and $\Box$ makes these logics able to express some interactions between the notion of necessity and truth. Let us observe that, at first glance, external connectives could look very much  like the ``statability'' operators introduced by Correia in \cite{Correia} for the (internal) modal version of PWK. However, the substantial difference is that statability operators are actually modal operators, while external connectives work at the propositional level. Nevertheless, the language of our logics is rich enough to allow the defition of (at least) one of them, namely Correia's statability operator $\mathcal{S}$ can be defined as $\mathcal{S}(\varphi)\coloneqq \Box + \varphi$.

The logics $\B^{\square}$ and $\MPWK$ have been presented syntactically via Hilbert-style systems and provided with a possible worlds semantics in terms of three-valued Kripke models. Completeness and decidability of the axiomatic calculi have been established, the former via Henkin-style proofs, the latter by the filtration technique.

As it is customary in modal logic, some standard axiomatic extensions of the logics have been presented. In the case of $\MPWK$ this is done proving that the well-known formulas T,4, and 5 correspond to their usual properties on the accessibility relation. On $\B$ we introduced the counterparts of those formulas in the external language and prove their correspondence with frame properties. The idea of considering these extensions goes in the direction of exploring \emph{epistemic} logics that could be useful to formalize notions such as knowledge and ignorance in a non-classical context (see e.g. \cite{KnowledgeinBelnap-Dunn}, \cite{Ignoranzasevera}). A more general question about if and how standard modal formulas defining frame properties can be transferred (see \cite{bilkova_frittella23}) to $\B^{\square}$ and $\MPWK$ obtaining classes of frames characterized by the same properties is an interesting topic yet to be explored.

In the paper the algebraizability of $\PWKe$ has been proved as well. This seemingly side result is the starting point for a further work on the algebraic semantics for $\B^{\square}$ and $\MPWK$. Together with the algebraizability of $\B$, the result of the current paper completes the strong correspondence between external weak Kleene logics and their algebraic counterpart, the class of Bochvar algebras. 

The next step is to explore the equivalent algebraic semantics of the global\footnote{We move from the local logics introduced in this paper to the global ones because it is a well-known result that the local normal modal logics based on $\mathrm{S5}$ and weaker systems are not algebraizable, while their global versions are not only algebraizable but even implicative (see e.g. \cite{FontBook}, Example 3.61).} versions of $\B^{\square}$ and $\MPWK$, which will result in a class of modal Bochvar algebras. These global logics together with the structure and properties of modal Bochvar algebras will be the focus of a future paper.

Finally, let us emphasize that we deliberately chose to work with weak Kleene logics as they are ``traditionally'' intended, namely the consequence relations defined via either truth or non-falsity preservation. A different choice to be explored in the future consists in modalizing logics of ``mixed'' nature, such as those where a non-false consequence can follow from true (only) premises (see e.g. \cite{ChemlaEgre17}), a possibility that is briefly discussed at the end of \cite{Correia}.

\section{Appendix: Algebraizability of $\PWKe$}\label{sec: Appendice}

Recall from Definition \ref{def.: logiche esterne} that $\mathsf{PWK}_e$ is the logic induced by the matrix $\langle \textbf{WK}^e, \lbrace 1, \ant \rbrace \rangle$. A Bochvar algebra is an algebra $\A = \langle A,\lor,\land,\neg,\Jdue,0,1\rangle$ of type $\langle 2,2,1,1,0,0\rangle$ satisfying the following identities and quasi-identities.\footnote{Bochvar algebras were originally introduced by Finn and Grigolia \cite{FinnGrigolia} in an extended language including $\Juno $ and $\Jzero$. These two operations are term-definable in the language $\langle A,\lor,\land,\neg,\Jdue,0,1\rangle$, as already explained in Section \ref{sec: external logics}. Besides, we are using here the much shorter equivalent axiomatization introduced in \cite[Theorem 7]{SMikBochvar}.}
\begin{enumerate}
\item $\varphi\vee \varphi\thickapprox \varphi$; \label{BCA:1}
\item$\varphi\lor \psi \thickapprox \psi \lor \varphi$; \label{BCA:2}
\item $(\varphi\lor \psi)\lor \delta\thickapprox  \varphi\lor(\psi\lor \delta)$; \label{BCA:3}
\item $\varphi\land(\psi\lor \delta)\thickapprox(\varphi\land \psi)\lor(\varphi\land \delta)$; \label{BCA:4}
\item $\neg(\neg \varphi)\thickapprox \varphi$; \label{BCA:5}
\item $\neg 1\thickapprox 0$; \label{BCA:6}
\item $\neg( \varphi\lor \psi)\thickapprox \neg \varphi\land\neg \psi$; \label{BCA:7}
\item $0\vee \varphi\thickapprox \varphi$; \label{BCA:8}
\item $\Jzero\Jdue \varphi\thickapprox \neg\Jdue \varphi$; \label{BCA:10}
\item $\Jdue\varphi \thickapprox\neg(\Jzero\varphi\vee \Juno\varphi)$; \label{BCA:13}
\item $\Jdue\varphi\vee\neg \Jdue\varphi \thickapprox 1$; \label{BCA:14}
\item $J_{_2}(\varphi\lor \psi)\thickapprox ( J_{_2}\varphi\land J_{_2}\psi)\vee( J_{_2}\varphi\land J_{_2}\neg \psi)\vee ( J_{_2}\neg \varphi\land J_{_2}\psi) $; \label{BCA:18}
\item $\Jzero \varphi \thickapprox \Jzero \psi \;\&\; \Jdue \varphi \thickapprox \Jdue \psi \;\Rightarrow\; \varphi \thickapprox \psi$.\label{BCA:quasi}
\end{enumerate}

\vspace{5pt}
Bochvar algebras form a proper quasi-variety, which we denote by $\class{BCA}$, and that is generated by $\textbf{WK}^e$. It plays the role of equivalent algebraic semantics of both Bochvar external logic \cite[Theorem 35]{Ignoranzasevera} and PWK external logic, as we show in Theorem \ref{th: algebrizzabilitÃ  PWKe}. Recall that a logic $\vdash_{\mathsf{L}}$ is algebraizable (\cite[Definition 3.11, Proposition 3.12]{FontBook}) with respect to a certain class of algebras $\class{K}$ if there exist two maps $\tau\colon Fm\to \mathcal{P}(Eq)$, $\rho\colon Eq\to \mathcal{P}(Fm)$ from formulas to sets of equations and from equations to sets of formulas such that:
\[
\text{(ALG1)}\;\;\;\;\;\; \Gamma\vdash\varphi\iff\tau[\Gamma]\vDash_{\mathsf{K}}\tau(\varphi),\] 
and
\[
\text{(ALG4)}\;\;\;\;\;\; \varphi\thickapprox\psi\Dashv\vDash_{\mathsf{K}}\tau\circ \rho(\varphi\thickapprox\psi).\]

As a notational convention, in the following proof we will indicate by $ Hom(\mathbf{Fm}, \WKt) $ the set of homomorphisms from the formula algebra (in the language of $\class{BCA}$) in $\WKt$.

\begin{theorem}\label{th: algebrizzabilitÃ  PWKe}
	$\PWKe$ is algebraizable w.r.t. $\class{BCA}$ with transformers $\tau(\varphi) := \lbrace\neg\Jzero\varphi\approx 1\rbrace$ and $\rho(\epsilon\approx\delta) := \lbrace\epsilon\equiv\delta\rbrace$.
\end{theorem}

\begin{proof}

 In our case (ALG1) and (ALG4) translate into:
	
	\begin{itemize}
		\item[(ALG1)] $\Gamma \vdash_{\mathsf{PWK}_e} \varphi \Leftrightarrow \tau [\Gamma] \vDash_{_\class{BCA}} \tau (\varphi)$,
		
		\item[(ALG4)] $\varepsilon \approx \delta \Dashv \vDash_{_\class{BCA}} \tau (\rho (\varepsilon \approx \delta))$.
	\end{itemize}

Moreover, since $\class{BCA}$ is the quasi-variety generated by $\WKt$, the above claims amount to the following:
	
	\begin{itemize}
		\item[(ALG1)] $\Gamma \vdash_{\mathsf{PWK}_e} \varphi \Leftrightarrow \tau [\Gamma] \vDash_{\WKt} \tau (\varphi)$,
		
		\item[(ALG4)] $\varepsilon \approx \delta \Dashv \vDash_{\WKt} \tau (\rho (\varepsilon \approx \delta))$.
	\end{itemize}
	
	(ALG1) ($\Rightarrow$) Suppose $\Gamma \vdash_{\mathsf{PWK}_e} \varphi$. Take $h\in Hom(\mathbf{Fm}, \WKt)$ s.t. $h(\neg\Jzero\gamma)=h(1)=1$ for every $\gamma\in\Gamma$, which implies $\neg\Jzero h(\gamma)=1$, i.e. $h(\gamma)\neq 0$. Since $\Gamma \vdash_{\mathsf{PWK}_e} \varphi$, by Definition \ref{def.: logiche esterne} it holds $h(\varphi)\neq 0$, hence $h(\neg\Jzero\varphi)=1=h(1)$. Thus, we have shown that $\tau [\Gamma] \vDash_{\WKt} \tau (\varphi)$.
 
 \noindent
 ($\Leftarrow$) Suppose $\tau [\Gamma] \vDash_{\WKt} \tau (\varphi)$ and let $h\in Hom(\mathbf{Fm}, \WKt)$ be s.t. $h(\gamma)\neq 0$ for every $\gamma\in \Gamma$. Therefore, by the hypothesis, $h(\neg\Jzero\varphi)=h(1)=1$, which implies $h(\varphi)\neq 0$, giving the desired conclusion.
 
	
	
	\vspace{0.2cm}
	
	(ALG4) ($\Rightarrow$) Consider an arbitrary identity $\varepsilon\approx\delta$ in the language of $\class{BCA}$. A simple calculation shows that $\tau (\rho (\varepsilon \approx \delta))$ is $\neg\Jzero(\varepsilon\equiv\delta)\approx 1$. Let $h\in Hom(\mathbf{Fm},\WKt)$ be s.t. $h(\varepsilon) = h(\delta)$, therefore $h(\varepsilon\equiv\delta)=1$. This implies $h(\neg\Jzero(\varepsilon\equiv\delta))=1=h(1)$, so we conclude $\varepsilon\approx\delta\models_{\WKt}\tau (\rho (\varepsilon \approx \delta))$.

	\noindent
	$(\Leftarrow)$ Suppose $h(\neg\Jzero(\varepsilon\equiv\delta))=\neg\Jzero(h(\varepsilon)\equiv h(\delta))=1 $, for $h \in Hom(\mathbf{Fm},\WKt)$. Now $\Jzero (h(\varepsilon) \equiv h(\delta)) = 0$ implies $(\Jdue h(\varepsilon)\leftrightarrow \Jdue h(\delta))\land(\Jzero h(\varepsilon)\leftrightarrow\Jzero h(\delta))\in\lbrace 1,\ant\rbrace$, but since $(\Jdue h(\varepsilon)\leftrightarrow \Jdue h(\delta))\land(\Jzero h(\varepsilon)\leftrightarrow\Jzero h(\delta))$ is an external formula it must be that $(\Jdue h(\varepsilon)\leftrightarrow \Jdue h(\delta))\land(\Jzero h(\varepsilon)\leftrightarrow\Jzero h(\delta)) = 1$. The fact that $\Jdue h(\varepsilon)\leftrightarrow \Jdue h(\delta)=1$ and $\Jzero h(\varepsilon)\leftrightarrow\Jzero h(\delta)=1$ implies $\Jdue h(\varepsilon)=\Jdue h(\delta)$ and $\Jzero h(\varepsilon)=\Jzero h(\delta)$. Applying the quasi-equation (13), we conclude $h(\varepsilon) = h(\delta)$.
 
\end{proof}

\textbf{Acknowledgments.}
We thank two anonymous referees for their detailed comments and observations, which allowed to improve the quality of the paper.\\
\noindent

\justifying
S. Bonzio acknowledges the support by the Italian Ministry of Education, University and Research through the PRIN 2022 project DeKLA (``Developing Kleene Logics and their Applications'', project code: 2022SM4XC8) and the PRIN Pnrr project ``Quantum Models for Logic, Computation and Natural Processes (Qm4Np)'' (cod. P2022A52CR). He also acknowledges the Fondazione di Sardegna for the support received by the projects GOACT (grant number
F75F21001210007) and MAPS (grant number F73C23001550007), the University of Cagliari for the support by the StartUp project ``GraphNet'', and also the support by the MOSAIC project (H2020-MSCA-RISE-2020 Project 101007627). Finally, he gratefully acknowledges also the support of the INDAM GNSAGA (Gruppo Nazionale per le Strutture Algebriche, Geometriche e loro Applicazioni). 

N. Zamperlin acknowledges the support of the MSCA-RISE Programme – Plexus: Philosophical, Logical, and Experimental Routes to Substructurality.


\end{document}